\documentclass[11pt]{amsart}  

\usepackage[pdftex, breaklinks, linktocpage=true, bookmarksopen=true,bookmarksopenlevel=0,bookmarksnumbered=true]{hyperref}

\usepackage{amsmath}
\usepackage{amsthm}
\usepackage{amsfonts}
\usepackage{amssymb}
\usepackage{yfonts}
\usepackage[all,cmtip]{xy}
\usepackage{amsxtra}
\usepackage{appendix} 
\usepackage{tensor} 
\usepackage{comment}
\usepackage{stmaryrd}

\usepackage[usenames,dvipsnames]{color}

\RequirePackage{color}
\definecolor{bwgreen}{rgb}{0.183,1,0.5}
\definecolor{bwmagenta}{rgb}{0.7,0.0,0.1}
\definecolor{bwblue}{rgb}{0.317,0.161,1}

\hypersetup{
pdftitle = {\texorpdfstring{Dieudonn\'e crystals and Wach modules for $p$-divisible groups}{Dieudonn\'e crystals and Wach modules for {p}-divisible groups}}
pdfauthor = {\textcopyright\ Bryden Cais and Eike Lau},
pdfcreator = {\LaTeX\ with package \flqq hyperref\frqq},
colorlinks = {true},
linkcolor={bwmagenta},	 
anchorcolor={},	 
citecolor={bwblue},	
filecolor={},	
menucolor={},	
runcolor={},	
urlcolor={bwgreen}, 
pageanchor={false},
}

\CompileMatrices

\DeclareFontFamily{OT1}{rsfs}{}
\DeclareFontShape{OT1}{rsfs}{n}{it}{<-> rsfs10}{}
\DeclareMathAlphabet{\mathscr}{OT1}{rsfs}{n}{it}

\DeclareFontFamily{OT1}{pzc}{}
\DeclareFontShape{OT1}{pzc}{n}{it}{<->s*[2.2]pzc}{}
\DeclareMathAlphabet{\mathpzc}{OT1}{pzc}{b}{sl}

\makeatletter

\newcommand{\Rmnum}[1]{\expandafter\@slowromancap\romannumeral #1@}
\makeatother

\DeclareMathOperator{\depth}{depth}

\DeclareMathOperator{\id}{id}

\DeclareMathOperator{\Frac}{Frac}

\DeclareMathOperator{\Hom}{Hom}
\DeclareMathOperator{\Ker}{Ker}
\DeclareMathOperator{\End}{End}

\DeclareMathOperator{\Gal}{Gal}

\DeclareMathOperator{\Spec}{Spec}
\DeclareMathOperator{\Spf}{Spf}

\DeclareMathOperator{\coker}{coker}

\DeclareMathOperator{\Fr}{Fr}

\DeclareMathOperator{\Cris}{Cris}
\DeclareMathOperator{\cris}{cris}

\DeclareMathOperator{\Lie}{Lie}

\DeclareMathOperator{\nr}{nr}
\DeclareMathOperator{\Mod}{Mod}
\DeclareMathOperator{\Fil}{Fil}
\DeclareMathOperator{\Filone}{Fil}

\DeclareMathOperator{\DF}{DF}
\DeclareMathOperator{\KR}{KR}
\DeclareMathOperator{\Dcat}{D}

\DeclareMathOperator{\Rep}{Rep}

\DeclareMathOperator{\sep}{sep}

\DeclareMathOperator{\im}{im}

\DeclareMathOperator{\Rad}{Rad}
\DeclareMathOperator{\Tor}{Tor}

\DeclareMathOperator{\pdiv}{pdiv}

\DeclareMathOperator{\pd}{pd}

\newcommand{\DDD}{\mathscr{D}}

\newcommand{\EEE}{\mathcal{E}}

\newcommand{\III}{\mathscr{I}}
\newcommand{\LLL}{\mathscr{L}}

\newcommand{\OOO}{\mathcal{O}}

\newcommand{\RRR}{\mathcal{R}}

\newcommand{\UUU}{\mathcal{U}}

\newcommand{\Fm}{\mathfrak{m}}

\newcommand{\FM}{\mathfrak{M}}
\newcommand{\FN}{\mathfrak{N}}
\newcommand{\FS}{\mathfrak{S}}

\newcommand{\EE}{{\mathbf{E}}}
\newcommand{\FF}{{\mathbb{F}}}
\newcommand{\GG}{{\mathbf{G}}}
\newcommand{\II}{{\mathbb{I}}}
\newcommand{\MM}{{\mathbb{M}}}

\newcommand{\QQ}{{\mathbf{Q}}}

\newcommand{\WW}{{\mathbb{W}}}
\newcommand{\ZZ}{{\mathbf{Z}}}

\newcommand{\nn}{{}_n}
\newcommand{\rr}{{}_r}

\newcommand*{\Z}{\ensuremath{\mathbf{Z}}}

\newcommand*{\s}{\mathfrak{S}}

\newcommand*{\C}{\mathbf{C}}
     
\newcommand*{\F}{\mathbf{F}}
\newcommand*{\scrF}{\mathscr{F}}

\newcommand*{\G}{\mathcal{G}}

\renewcommand*{\O}{\mathscr{O}}

\newcommand*{\scrHom}{\mathscr{H}\mathit{om}}

\newcommand*{\scrLie}{\mathscr{L}\mathit{ie}}      
\newcommand*{\scrD}{\mathscr{D}}          

\newcommand*{\D}{\ensuremath{\mathbf{D}}}

\renewcommand*{\H}{\ensuremath{\mathfrak{H}}}

\renewcommand*{\int}{\ensuremath{\mathrm{int}}}

\renewcommand*{\u}[1]{\underline{#1}}

\newcommand*{\wh}[1]{\widehat{#1}}

\renewcommand{\tilde}{\widetilde}
\renewcommand{\bar}{\overline}

\DeclareMathOperator{\Win}{Win}
\DeclareMathOperator{\BT}{BT}

\theoremstyle{plain}
	\newtheorem{thm}{Theorem}
  \newtheorem{theorem}{Theorem}
  \newtheorem{proposition}[theorem]{Proposition}
  \newtheorem{lemma}[theorem]{Lemma}
  \newtheorem{corollary}[theorem]{Corollary}

\theoremstyle{definition}
  \newtheorem{definition}[theorem]{Definition}

    \newtheorem{hypothesis}[theorem]{Hypothesis}
  
  \newtheorem{Example}[theorem]{Example}
  
\theoremstyle{remark}
  \newtheorem{example}[theorem]{Example}
  \newtheorem*{rem}{Remark}
  \newtheorem{remark}[theorem]{Remark}
  \newtheorem{remarks}[theorem]{Remarks}

  \newtheorem{point}[theorem]{}
   	
\include{header}

\numberwithin{theorem}{subsection}  

\begin{document}
\bibliographystyle{plain}

\title[Dieudonn\'e crystals and Wach modules for $p$-divisible groups]{Dieudonn\'e crystals and Wach modules for $p$-divisible groups}

\author{Bryden Cais}
\address{Bryden Cais\\
Department of Mathematics\\
University of Arizona\\
Tucson, Arizona 85721\\
USA}
\email{cais@math.arizona.edu}

\author{Eike Lau}
\address{Eike Lau\\
Fakult\"at f\"ur Elektrotechnik, Informatik und Mathematik\\
Institut f\"ur Mathematik\\
Warburger Str. 100\\
33098 Paderborn\\
Germany}
\email{elau@math.uni-paderborn.de}

\thanks{Acknowledgements to be added}


\date{April 19, 2016}

\begin{abstract}
	Let $k$ be a perfect field of characteristic $p>2$ and $K$ an
	extension of $F=\Frac W(k)$ contained in some $F(\mu_{p^r})$.
	Using crystalline Dieudonn\'e theory, we 
	provide a classification of $p$-divisible groups over $R=\OOO_K[[t_1,\ldots,t_d]]$
	in terms of finite height $(\varphi,\Gamma)$-modules over 
	$\FS:=W(k)[[u,t_1,\ldots,t_d]]$.     
	When $d=0$, such a classification is a consequence
	of (a special case of) the theory of Kisin--Ren; 
	in this setting, our construction
	gives an independent proof of this result, and moreover allows us to recover
	the Dieudonn\'e crystal of a $p$-divisible group 
	from the Wach module associated to its Tate module
	by Berger--Breuil or by Kisin--Ren. 
\end{abstract}

\maketitle

\vspace{-0.3in}

\tableofcontents

\section{Introduction}

\numberwithin{equation}{section}

Let $k$ be a perfect field of characteristic $p>2$ and let $K$ be a finite extension 
of $F=\Frac W(k)$ contained in $F_r=F(\mu_{p^r})$ for a fixed $r\ge 1$. 
Fix $d\ge 0$ and write $R=\OOO_K[[t_1,\ldots,t_d]]$
for the ring of power series in $d$ variables over the valuation ring of $K$.
Our main result is the following description of the category $\pdiv(R)$
of $p$-divisible groups over $R$ in terms of $(\varphi,\Gamma)$-modules.

Put $F_\infty=F(\mu_{p^{\infty}})$ and let $\Gamma_K=\Gal(F_\infty/K)$.
The ring $\FS=W(k)[[u,t_1,\ldots,t_d]]$ is equipped with the cyclotomic Frobenius
$\varphi(1+u)=(1+u)^p$ and $\varphi(t_i)=t_i^p$, and with an action of $\Gamma_F$ by 
$\gamma(1+u)=(1+u)^{\chi(\gamma)}$ where $\chi$ is the $p$-adic cyclotomic character,
while $\gamma(t_i)=t_i$.
We define $E=\varphi^r(u)/\varphi^{r-1}(u)$ and 
denote by $\BT(\u\FS)_{\Gamma_K}$ the category of finite free
$(\varphi,\Gamma_K)$-modules $\FM$ over $\FS$ such that the $\FS$-module
generated by $\varphi(\FM)$ contains $E\FM$ and $\Gamma_K$ acts trivially on $\FM/u\FM$.

\begin{thm}
\label{Th:Main}
There is a contravariant equivalence of categories
between $\pdiv(R)$ and $\BT(\u\FS)_{\Gamma_K}$
sending $G$ to $\FM(G)$.
\end{thm}

Explicitly, this is proved in Corollary \ref{Co:Final}; 
see Proposition \ref{Pr:Final} for a slightly more general statement.
The equivalence will be derived from crystalline Dieudonn\'e theory, 
which allows to recover from $\FM(G)$ the Dieudonn\'e crystal of $G$.
Moreover it is compatible with duality and with change of $r$ and $K$ and also 
allows to recover from $\FM(G)$ the Tate module of $G$ as a representation of
the absolute Galois group of $\Frac R[t_1^{p^{-\infty}},\ldots,t_d^{p^{-\infty}}]$;
see Proposition \ref{Pr:Comp}.

In the case $d=0$ (so $R=\OOO_K$),
using Kisin's result \cite{KisinFcrystal} that every crystalline representation of
the absolute Galois group $\G_K$ 
with Hodge-Tate weights $0$ and $1$ comes from a $p$-divisible group over $\OOO_K$,
Theorem \ref{Th:Main} is a consequence of the description of lattices 
in crystalline representations of $\G_K$ in terms of Wach modules
by Wach \cite{Wach96}, Colmez \cite{Colmez:Crelle99}, Berger-Breuil \cite{BB}, 
or in terms of the modules of Kisin-Ren \cite{KisinRen}.
In fact, the latter generalise a variant of the former to the case of a 
Lubin-Tate extension in place of the cyclotomic extension,
and our modules $\FM(G)$ essentially
coincide with those of \cite[Corollary 3.3.8]{KisinRen}; 
see \S\ref{FinEHt}.

However, even for $d=0$, our approach gives more than just an independent proof
of this fact as it provides a link between Kisin-Ren or Wach modules and Dieudonn\'e crystals.
This can be important in applications, as for example in \cite{CaisHida2}.

\medskip

Let us outline the proof of Theorem \ref{Th:Main}.
For simplicity we assume that $K=F_r$, so $E$ is the minimal polynomial
of a prime element of $\OOO_K$.
Let $S$ be the $p$-adic completion of the divided power envelope of the ideal 
$E\FS$ of $\FS$.
We denote by $\Win(\u S)_{\Gamma_K}$ 
the category of strongly divisible modules $M$
in the sense of Breuil \cite{Breuil:Groupes}, called windows by Zink \cite{Zink:Windows},
equipped with an action of $\Gamma_K$ that is trivial on $M\otimes_{S}W(k)$.
There are natural functors
\begin{equation}
\label{Eq:Intro-Funct-1}
\pdiv(R)\longrightarrow\Win(\u S)_{\Gamma_K}\longleftarrow \BT(\u\FS)_{\Gamma_K},
\end{equation}
the first given by evaluating the Dieudonn\'e crystal of a $p$-divisible group,
and the second by a base change operation 
$\FM\mapsto M=\FM\otimes_{\FS,\varphi}S$.
The second functor in \eqref{Eq:Intro-Funct-1} is an equivalence by known results.
We will show that the first functor is an equivalence too,
which proves Theorem \ref{Th:Main}.

Let us go into a little more detail.
Without the action of $\Gamma_K$, the construction
makes sense for {\em any} $d+1$-dimensional 
complete regular local ring $R'$ of characteristic zero and residue field $k$.
Writing $R'$ as a quotient of the ring $\FS'=W(k)[[u_0,\ldots,u_{d}]]$
and equipping $\FS'$ with an arbitrary Frobenius lift $\varphi$,
we obtain functors
\begin{equation}
\label{Eq:Intro-Funct-2}
\pdiv(R')\longrightarrow\Win(\u S')\longleftarrow \BT(\u\FS')
\end{equation}
where $S'$ is again the $p$-adic completion of the divided power envelope of
the kernel of $\FS'\to R'$.
The second functor in \eqref{Eq:Intro-Funct-2}
is always an equivalence by Proposition \ref{Pr:lambda-crys} below; 
when $d=0$ ({\em i.e.}~$R'=\OOO_{K'}$ for a totally ramified extension $K'$ of $F$) 
this is proved by Caruso-Liu \cite{Caruso-Liu:Quasi}, and the case
$R'=\OOO_{K'}[[t_1,\ldots,t_d]]$ is covered by Kim \cite{Kim:FormallySmooth}.

We prove in Proposition \ref{Pr:BT2Win} below that the first functor in \eqref{Eq:Intro-Funct-2} is an equivalence when the divided
derivative $p^{-1}d\varphi$ of the Frobenius lift $\varphi$ on $\FS'$ is topologically nilpotent, 
which is for example satisfied when $\varphi(u_i)=u_i^p$, 
but not in the setting of Theorem \ref{Th:Main} with $\varphi(1+u)=(1+u)^p$.
In the case $R'=\OOO_{K'}$ and $\varphi(u)=u^p$, 
this equivalence is due to Breuil \cite{Breuil:Groupes};
the resulting equivalence between the categories $\pdiv(\OOO_{K'})$ and $\BT(\u\FS')$
was first proved by Kisin \cite{KisinFcrystal} using a different route.

In order to use  
the equivalence of Proposition \ref{Pr:BT2Win}
we have to switch between different choices of $\varphi$.
We consider the category $\DF(R')$ 
of Dieudonn\'e crystals over $\Spf R'$ equipped with an admissible
Hodge filtration as in Definition \ref{Def:DF} below.
For every choice of $\varphi$,
this category is equivalent to the category $\Win(\u S')^\nabla$ of windows with a connection. 
If the divided derivative $p^{-1}d\varphi$ is topologically nilpotent, the
forgetful functor $\Win(\u S')^\nabla\to\Win(\u S')$ is an equivalence.
It follows that the crystalline Dieudonn\'e functor $\pdiv(R')\to\DF(R')$ 
is an equivalence because both categories are equivalent to $\Win(\u S')$ when $\varphi(u_i)=u_i^p$.
This implies that for every $\varphi$, the functor $\pdiv(R')\to\Win(\u S')^\nabla$ 
is an equivalence.

It remains to see that in the ``cyclotomic'' situation of Theorem \ref{Th:Main}, 
the category $\Win(\u S)_{\Gamma_K}$ is equivalent to $\Win(\u S)^\nabla$.
The essential point is that $\nabla$ is determined by the differential operator 
$N=\nabla_{\partial/\partial u}$,
and the differential at the identity of an action of $\Gamma_K$ gives $N$.
This construction is well-known in the context of $(\varphi,\Gamma)$-modules
over the Robba ring, see \cite[\S4.1]{Berger:RepDiff}, 
but additional care is required in the present situation as we must 
control the denominators of $p$ which occur.

While the relation between the module $\FM(G)$ and the Dieudonn\'e crystal of $G$ is clear
from the construction, to recover the Tate module of $G$ from $\FM(G)$
we use a variant of Faltings' integral comparison isomorphism for $p$-divisible groups 
\cite[Theorem 7]{Faltings}.
It is then straightforward to relate $\FM(G)$ to the modules
of Kisin-Ren \cite{KisinRen} and Berger-Breuil \cite{BB} when $d=0$.


\numberwithin{equation}{subsection}

\section{\texorpdfstring{Frames, windows, and Dieudonn\'e crystals}{Frames, windows, and Dieudonn\'e
 crystals}}\label{FramesEtc}

Let $p$ be a prime.
In this section, we recall and elaborate on the relation between Dieudonn\'e crystals,
Breuil modules, and Kisin modules associated to $p$-divisible groups following 
\cite{BBM, Breuil:Groupes, deJong:Crystalline,  KisinFcrystal, Kim:FormallySmooth,
Lau:Dieudonne, Zink:Windows}. Most of this is well-known, but some aspects
appear to be new.
Technically we will use the notion of frames and windows, which we recall first.

\subsection{Frames and windows}\label{FrameRecall}

Frames and windows were introduced by Zink \cite{Zink:Windows} and generalised in \cite{Lau:Frames}. We use the definition of \cite{Lau:Frames} with some minor modifications.

\begin{definition}
\label{Def:Frame}
A \emph{frame} $\u{S}:=(S,\Filone S,R,\varphi,\varphi_1,\varpi)$ consists of a 
ring $S$, an ideal $\Filone S$ of $S$, the quotient ring $R:=S/\Filone S$, a ring endomorphism
$\varphi:S\rightarrow S$ reducing to the $p$-power map modulo $pS$, 
a $\varphi$-linear map $\varphi_1:\Filone S\rightarrow S$, and an element
$\varpi\in S$ such that $\varphi=\varpi\varphi_1$ on $\Filone S$.
We further require that $\Filone S+ pS \subseteq \Rad(S)$.
We call $\u S$ a {\em lifting frame} if in addition every finite projective $R$-module 
lifts to a finite projective $S$-module.
\end{definition}

\begin{remarks}
\label{Re:Frame}
In \cite{Lau:Frames} the following \emph{surjectivity condition} is also imposed:
The image of $\varphi_1$ generates the unit ideal of $S$.
Then the element $\varpi$ is determined by the rest of the data, namely if $1=\sum a_i\varphi_1(b_i)$
then $\varpi=\sum a_i\varphi(b_i)$. The surjectivity condition is satisfied in many examples,
but is not necessary in the theory; cf.\ \S 11 of the arxiv version of \cite{Lau:Frames}.
It is not present in Zink's notion of frames; see Definition \ref{Def:PDframe} and 
Remark \ref{Re:PDframe} below.
If the surjectivity condition holds we will sometimes omit $\varpi$ from the notation
and write $\u S=(S,\Filone S,\varphi,\varphi_1)$
in order to be consistent with the literature.

In this paper, we will only consider lifting frames.
All frames with $R$ local are lifting frames 
(as projective $R$-modules of finite type are free).
The condition $\Filone S+pS\subseteq \Rad(S)$ is automatic 
if $S$ is local and $p\in R$ is a non-unit.
\end{remarks}

\begin{definition}
\label{Def:Window}
A {\em window} $\u{M}:=(M,\Filone M,\Phi,\Phi_1)$ over a frame $\u{S}$
consists of a projective $S$-module $M$ of finite type,
an $S$-submodule $\Filone M \subseteq M$ and
$\varphi$-linear maps $\Phi:M\rightarrow M$ and $\Phi_1:\Filone M\rightarrow M$
such that:
\begin{enumerate}
	\item There exists a decomposition of $S$-modules 
	$M=L\oplus N$ with $\Filone M=L\oplus(\Filone  S)N$.
	\label{normaldecomp}
	\item If $s\in \Filone  S$ and $m\in M$ then $\Phi_1(sm) = \varphi_1(s)\Phi(m)$.\label{phirel}
	\item If $m\in\Filone M$ then $\Phi(m)=\varpi\Phi_1(m)$. \label{phirel2}
	\item $\Phi_1(\Filone M)+\Phi(M)$ generates $M$ as an $S$-module.\label{surj}
\end{enumerate}
Naturally, a homomorphism of windows is an $S$-linear map 
that preserves the filtration and commutes with $\Phi$ and with $\Phi_1$. 
We write $\Win(\u S)$ for the category of $\u S$-windows.
A short sequence of windows is called exact if the sequences of $M$'s and
of $\Filone M$'s are both exact.
\end{definition}

\begin{remarks}
\label{Re:Window}
If $\u S$ satisfies the surjectivity condition of Remarks \ref{Re:Frame},
then \eqref{phirel2} follows from \eqref{phirel}, 
condition \eqref{surj} means that $\Phi_1(\Filone M)$ generates $M$,
and the map $\Phi$ is determined by $\Phi_1$. 

When $\u{S}$ is a lifting frame, then (\ref{normaldecomp})
is equivalent to the simultaneous requirement that $(\Filone  S) M \subseteq \Filone  M$ 
and that $M/\Filone M$ is projective as an $R$-module.

A decomposition as in (\ref{normaldecomp}) is called a {\em normal decomposition}.
If $(M,\Filone M)$ and a normal decomposition $M=L\oplus N$ are given,
the set of pairs $(\Phi,\Phi_1)$ which define a window $\u M$ is in bijection
with the set of $\varphi$-linear isomorphisms $\Psi:L\oplus N\to M$ via
$\Psi(l+n):=\Phi_1(l)+\Phi(n)$.

Let us write $\varphi^*M:=\FS\otimes_{\varphi,\FS} M$
for the scalar extension of $M$ along $\varphi:\FS\rightarrow \FS$,
and denote by $F:\varphi^*M\rightarrow M$ the linearization of $\Phi$
determined by $F(s\otimes x)=s\Phi(x)$.  
There is a unique $S$-linear map 
$V:M\rightarrow \varphi^* M$
with $V(\Phi_1(m))=1\otimes m$ for $m\in\Fil M$ and $V(\Phi(m))=\varpi\otimes m$
for $m\in M$; here the last condition is automatic if $\u S$ satisfies the surjectivity
condition. The composition of $F$ and $V$
in either order is multiplication by $\varpi$. See \cite[Lemma 2.3]{Lau:Dieudonne}.
\end{remarks}

\begin{definition}\label{Def:FrameHom}
A {\em homomorphism of frames} $\alpha:\u{S}\to\u{S}'$ 
is a homomorphism of rings $\alpha: S\rightarrow S'$ 
that intertwines $\varphi$ with $\varphi'$ and carries $\Filone S$ into $\Filone S'$ 
and which satisfies
$\varphi'_1\alpha=c\cdot \alpha \varphi_1$ and $\alpha(\varpi)=c\varpi'$ for a unit $c\in S'$,
which is part of the data.
If we wish to specify $c$, we will say that $\alpha$ is a {\em $c$-homomorphism}.
A {\em strict} homomorphism of frames is simply a $1$-homomorphism.
\end{definition}

\begin{remark}\label{Rem:cdetermined}
If $\u S$ satisfies the surjectivity condition, then $c$ is uniquely determined by the relation
$\varphi_1'\alpha=c\cdot\alpha\varphi$, and the relation $\alpha(\varpi)=c\varpi'$ is a consequence.
\end{remark}

	Let $\alpha: \u{S}\rightarrow \u{S}'$ be a $c$-homomorphism of frames.
	If $\u M$ and $\u M'$ are windows over $\u S$ and $\u S'$,
	then an {\em $\alpha$-homomorphism of windows} $f:\u M\to\u M'$
	is a filtration-compatible homomorphism of $S$-modules $f:M\rightarrow M'$ 
	that intertwines $\Phi$ with $\Phi'$ and satisfies
	$\Phi_1' f = c\cdot f\Phi_1 $ on $\Filone  M$.	 
	Thus the morphisms in $\Win(\u S)$ are the $\id_{\u{S}}$-homomorphisms in this sense.
	There is a {\em base change functor}
\begin{equation}
	\xymatrix{
		{\alpha^*:\Win({\u{S}})} \ar[r] & {\Win({\u{S}'})}
		}\label{bcfunc}
\end{equation}
which is characterized by the universal property
\begin{equation}
	\Hom_{\u{S}'}(\alpha^*\u{M},\u{M}') = \Hom_{\alpha}(\u{M},\u{M}')\label{UnivPropBC}
\end{equation}
for $\u{M}\in \Win({\u{S}})$ and $\u{M}'\in \Win({\u{S}'})$;
here we write $\Hom_{\alpha}(\u{M},\u{M}')$ for the set of $\alpha$-homomorphisms.
Explicitly, if $\u M=(M,\Filone M,\Phi,\Phi_1)$ then the base change
$\alpha^*\u M=(M',\Filone M',\Phi',\Phi_1')$ has
$M'=S'\otimes_SM$ with $\Filone  M'$ the $S'$-submodule of $M'$ generated by
$(\Filone  S')M'$ and the image of $\Filone  M$.
The maps $\Phi'$ and $\Phi_1'$ are uniquely determined
by the requirement that $M\to M'$, $m\mapsto 1\otimes m$
is an $\alpha$-homomorphism; in particular $\Phi'(s'\otimes m)=\varphi'(s')\otimes\Phi(m)$. 
See \cite[\S2.1]{Lau:Dieudonne} or \cite[Lemma 2.9]{Lau:Frames}.
	The frame homomorphism $\alpha$ is called {\em crystalline}
	if the base change functor (\ref{bcfunc}) is an equivalence of categories.

\begin{definition}
The {\em dual} of an $\u{S}$-window $\u{M}$ is
$\u{M}^t:=(M^t,\Filone  M^t,\Phi^t,\Phi_1^t)$ where:
\begin{enumerate}
	\item $M^t:=\Hom_{S}(M,S)$ is the $S$-linear dual of $M$,
	\item $\Filone  M^t:=\{f\in M^t\ :\ f(\Filone  M)\subseteq \Filone  S\}$,
	\item \label{It:Phi1t}The maps $\Phi^t:M^t\to M^t$ and  $\Phi_1^t: \Filone  M^t\rightarrow M^t$ 
	are determined by the relations
	\begin{align*}
	\Phi_1^t(f)(\Phi_1(m)) & = \varphi_1(f(m))  &\text{for } & f\in\Filone M^t \text{ and } m\in\Filone M,  \\
	\Phi_1^t(f)(\Phi(m)) & = \varphi(f(m)) &\text{for }  & f\in\Filone M^t \text{ and } m\in M,  \\
	\Phi^t(f)(\Phi_1(m)) & = \varphi(f(m)) &\text{for }  & f\in M^t \text{ and } m\in\Filone M,  \\
	\Phi^t(f)(\Phi(m)) & = \varpi\varphi(f(m))  &\text{for } & f\in M^t \text{ and } m\in M.  
	\end{align*}
\end{enumerate} 
\end{definition}

\begin{remarks}
If the surjectivity condition of Remark \ref{Re:Frame} holds for $\u S$, then the first of the relations in \eqref{It:Phi1t} implies the others.

The existence of the pair of maps $(\Phi^t,\Phi_1^t)$ 
requires an argument using normal representations; see \cite[\S 2.1]{Lau:Dieudonne}
or \cite[\S2]{Lau:Frames}.

There is a canonical ``double duality" isomorphism
$\u{M}^{tt}\simeq\u M$ in $\Win(\u{S})$.

If $F:\varphi^*M\rightarrow M$ and $V:M\rightarrow \varphi^*M$
are the $S$-linear maps defined in Remark \ref{Re:Window}, then the corresponding
maps for $M^t$ are the $S$-linear duals of $V$ and $F$, respectively. 

The formation of duals commutes with base change along strict homomorphisms of frames,
and also along a $c$-homomorphism $\underline S\to\underline S'$
provided a unit $y\in S'$ with $c=y/\varphi'(y)$ is given:
there is always a canonical isomorphism $(\alpha^*\u M)^t\cong \alpha^*(\u M^t)_{c^{-1}}$
where the subscript means that $\Phi$ and $\Phi_1$ are multiplied by $c^{-1}$, and
multiplication by $y$ then gives an isomorphism $y:\alpha^*(\u M^t)_{c^{-1}}\to \alpha^*(\u M^t)$;
see \cite[Lemma 2.14]{Lau:Frames}.
\end{remarks}

\begin{example}
For every frame $\u S$ we have the windows $\u S=(S,\Filone S,\varphi,\varphi_1)$
and its dual $\u S^t=(S,S,\varpi\varphi,\varphi)$.
\end{example}

Let us recall the frames of \cite{Zink:Windows}, which we call {\em PD-frames} here:

\begin{definition}
	\label{Def:PDframe}
	Let $R$ be a $p$-adically complete ring.
	A {\em PD-frame for $R$} is a surjective ring homomorphism $S\to R$ together
	with a Frobenius lift $\varphi:S\to S$ where $S$ is $p$-adically complete without $p$-torsion
	such that the kernel of $S\to R$ is a PD-ideal.
\end{definition}

\begin{remark}
\label{Re:PDframe}
A PD-frame for $R$ satisfies $\varphi(\Filone S)\subseteq pS$ and therefore
extends uniquely to a frame $\u S=(S,\Filone S,R,\varphi,\varphi_1,p)$ 
in the sense of Definition \ref{Def:Frame} with
$\varphi_1:=p^{-1}\varphi$. 
Moreover $\u S$ is a lifting frame
because the kernel of $S/pS\to R/pR$ is a nilideal due to the divided powers.
We will also call the collection $\u S$ a PD-frame.
\end{remark}

Our second main example of frames is the following:

\begin{Example}
\label{Ex:frame-FS}
Assume that $\FS$ is a ring with a fixed non-zero divisor $E\in\FS$ such that
$p,E\in\Rad(\FS)$, and $\varphi:\FS\to\FS$ is an endomorphism 
that induces the $p$-power map on $\FS/p\FS$. 
We can form the frame
$$
\u\FS=(\FS,\Fil\FS,R,\varphi,\varphi_1')
$$
with $\Fil\FS=E\FS$ and $R=\FS/E\FS$ and $\varphi_1'(Ex)=\varphi(x)$ for $x\in\FS$.
The element $\varpi$ with $\varphi=\varpi\varphi_1'$ on $\Filone\FS$
is uniquely determined by $\varpi=\varphi(E)$ and is omitted from the notation;
cf.\ Remark \ref{Re:Frame}.
\end{Example}

\begin{definition}
\label{Def:KisinMod}
A Barsotti-Tate (BT) 
module over $\u\FS$ is a pair $(\FM,\varphi)$ where $\FM$ is a finite
projective $\FS$-module and $\varphi:\FM\to\FM$ is a $\varphi$-linear map
such that the cokernel of the associated linear map $1\otimes\varphi:\varphi^*\FM\to\FM$
is annihilated by $E$ and projective as an $R$-module.
The category of BT 
modules over $\u\FS$ is denoted by $\BT(\u\FS)$.
\end{definition}

\begin{remark}
If $R$ and $\FS$ are regular rings,
the projectivity of $M:=\coker(1\otimes\varphi)$ over $R$ is automatic:
Indeed, assume that $M\ne 0$.
Since $\depth_R(M)= \depth_{\FS}(M)$ and 
$\depth(R) = \dim(R) = \dim(\FS)-1 = \depth(\FS)-1$,
the Auslander-Buchsbaum formula gives $\pd_R(M) = \pd_{\FS}(M)-1$.
As $1\otimes\varphi$ becomes an isomorphism after
inverting $E$, it is injective, and it follows that 
$\varphi^*\FM\to\FM$ is a projective resolution of $\coker(1\otimes\varphi)$.
Thus $\pd_\FS(M)= 1$ and $\pd_R(M)=0$.
\end{remark}

\begin{remark}
\label{Rk:dual-BT}
One verifies that the category $\BT(\u\FS)$ can be equipped with the
following duality operation:
The dual of $(\FM,\varphi)$ is $(\FM^t,\varphi^t)$
where $\FM^t:=\Hom_{\FS}(\FM,\FS)$ and
$\varphi^t$ is determined by $\varphi^t(f)(\varphi(m))=E\varphi(f(m))$
for $f\in\FM^t$ and $m\in \FM$.
\end{remark}

\begin{lemma}
\label{Le:WinKisin}
If all finite projective $R$-modules lift to finite projective $\FS$-modules,
there is an equivalence of categories $\Win(\u\FS)\to\BT(\u\FS)$ given by
\[
\u M=(M,\Filone M,\Phi,\Phi_1)\mapsto (\Filone M,E\Phi_1).
\]
The equivalence preserves exactness and duality.
\end{lemma}

In many examples the hypothesis of Lemma \ref{Le:WinKisin} 
is satisfied because $\FS$ is local or $E$-adically complete.

\begin{proof}
This is standard.
Let $\u M$ be an $\u\FS$-window.
Then the $\FS$-module $\FM:=\Fil M$ is projective, and the linearization of 
$\Phi_1:\FM\to M$ is an isomorphism $\varphi^*\FM\cong M$. 
The inclusion $\FM\to M$ composed with the inverse of this isomorphism
defines a linear map $\psi:\FM\to\varphi^*\FM$, and there is
a unique linear map $\tilde\varphi:\varphi^*\FM\to\FM$ with $\tilde\varphi\psi=E\cdot\id$.
The cokernel of $\tilde\varphi$ is isomorphic to $\Fil M/EM$, which is projective over $R$.
Let $\varphi(x)=\tilde\varphi(1\otimes x)$ for $x\in\FM$. Then $(\FM,\varphi)\in\BT(\u\FS)$,
and one verifies that $\varphi=E\Phi_1$.

Conversely, for $(\FM,\varphi)\in\BT(\u\FS)$ there is a unique linear map
$\psi:\FM\to\varphi^*\FM$ with $\tilde\varphi\psi=E\cdot\id$, and we get a window
$\u M=(M,\Filone M,\Phi,\Phi_1)$ by setting $M:=\varphi^*\FM$ and $\Filone M:=\psi(\FM)$ 
and $\Phi_1(\psi(x)):=1\otimes x$ for $x\in\FM$ and $\Phi(x):=1\otimes\varphi(x)$ for $x\in M$.
\end{proof}

\begin{example}
\label{Ex:WinKisin}
The window $\u\FS=(\FS,E\FS,\varphi,\varphi_1')$ corresponds to the BT 
module $(\FS,\varphi)$, and its dual $\u\FS^t=(\FS,\FS,\varphi(E)\varphi,\varphi)$ 
corresponds to $(\FS,E\varphi)$.
\end{example}

\subsection{\texorpdfstring{Lifting windows modulo powers of $p$}{Lifting windows modulo powers of p}}
\label{Se:LiftWin}

As a preparation for Proposition \ref{Pr:lambda-crys} below, 
we consider the following situation.  Let 
$$
\u S'=( S',\Fil S', R',\varphi',\varphi_1',\varpi')
\xrightarrow{\;\lambda\;}\u S=(S,\Fil S,R,\varphi,\varphi_1,\varpi)
$$ 
be a frame homomorphism such that the rings $ S', R',S,R$ are 
$p$-adically complete and $\ZZ_p$-flat. 
For each $n\ge 1$ we obtain frames  $\u S\nn=\u S\otimes\ZZ/p^n\ZZ$ and 
$\u S'_n=\u S'\otimes\ZZ/p^n\ZZ$, the tensor product taken componentwise, 
and a frame homomorphism $\lambda_n:\u S'_n\to\u S\nn$.
Recall that a frame homomorphism is called crystalline if it induces an equivalence of the
associated window categories.

\begin{proposition}
\label{Pr:lift-pn}
If $\lambda_1$ is crystalline, then all $\lambda_n$ and $\lambda$ are crystalline.
\end{proposition}

\begin{proof}
For a fixed $\u S\nn$-window $\u M\nn=(M,\Filone  M,\Phi,\Phi_1)$ 
we write $\LLL(\u M\nn)$ the category of lifts of $\u M\nn$
to an $\u S{}_{n+1}$-window $\u{\tilde M}\nn=(\tilde M,\Filone \tilde M,\tilde\Phi,\tilde\Phi_1)$;
the morphisms in this category are isomorphisms that induce the identity of $\u M\nn$.
Let $\u{\bar M}=\u M\otimes\ZZ/p\ZZ=(\bar M,\Filone \bar M,\bar\Phi,\bar\Phi_1)$.
We claim that there is an equivalence of categories
\begin{equation}
\LLL(\u M\nn)\cong\EEE xt^1_{\u S_1}(\u{\bar M},\u{\bar M}),
\label{Eq:equiv-D-E}
\end{equation}
depending on the choice of a base point in $\LLL(\u M\nn)$.
All lifts to $\u S{}_{n+1}$ of the pair $(M,\Filone  M)$ are isomorphic, 
and we choose one of them $(\tilde M,\Filone \tilde M)$.
Let $D^1$ be the set of pairs $(G,G_1)$
of $\varphi$-linear maps $G:\bar M\to\bar M$ and $G_1:\Filone \bar M\to\bar M$ 
that satisfy the relations $G_1(ax)=\varphi_1(a)G(x)$ for $a\in\Filone  S_1$ and $x\in\bar M$
and $G(x)=\varpi G_1(x)$ for $x\in\Filone\bar M$.
The set of pairs $(\tilde\Phi,\tilde\Phi_1)$ which complete the lift $\u{\tilde M}$ 
is a principally homogeneous set under the abelian group
$D^1$, where $(G,G_1)$ acts on $(\tilde\Phi,\tilde\Phi_1)$
by adding $(p^nG,p^n G_1)$.
The group of automorphisms of the chosen lift $(\tilde M,\Filone \tilde M)$
is isomorphic to $D^0:=\End(\bar M,\Filone \bar M)$
by sending $\alpha\in D_0$ to the automorphism $1+p^n\alpha$,
and the action of these automorphisms on the pairs $(\tilde F,\tilde F_1)$ is given by
$$
d:D^0\to D^1, \qquad d(\alpha)=(\bar\Phi\alpha-\alpha\bar\Phi,\bar\Phi_1\alpha-\alpha\bar\Phi_1).
$$
It follows that the category $\LLL(\u M\nn)$ is equivalent to the quotient groupoid $[D_1/D_0]$.
Similary, for an extension of $\u S_1$-windows 
$0\to\u{\bar M}\to\u N\to\u{\bar M}\to 0$,
the underlying pair of modules $(N,\Filone  N)$ is isomorphic to 
$(\bar M\oplus\bar M,\Filone \bar M\oplus\Filone \bar M)$.
The operators $\Phi_N$ and $\Phi_{1,N}$ of $\u N$ are then given by block matrices
$
\left(\begin{smallmatrix}\bar\Phi&G\\0&\bar\Phi\end{smallmatrix}\right)
$ 
and 
$
\left(\begin{smallmatrix}\bar\Phi_1&G_1\\0&\bar\Phi_1\end{smallmatrix}\right)
$
for some $(G,G_1)\in D_1$, the automorphism group of $(N,\Filone N)$
is isomorphic to $D^0$ by sending $\alpha\in D_0$ to the
automorphism $\left(\begin{smallmatrix}1&\alpha\\0&1\end{smallmatrix}\right)$,
and the action of automorphisms on pairs $(\Phi_N,\Phi_{1,N})$ is given by $d$
as above. The equivalence \eqref{Eq:equiv-D-E} follows.
For an $\u S'_n$-window $\u M'_n$ we get a similar
equivalence
$$
\LLL(\u M'_n)\cong\EEE xt^1_{\u S'_1}(\u{\bar M}{}',\u{\bar M}{}').
$$
We apply this for $\u M\nn=\lambda_n(\u M'_n)$ and deduce that 
since $\lambda_1$ is assumed to be crystalline, the functor
\begin{equation}
\label{Eq-win-lift}
\LLL(\u M'_n)\to\LLL(\u M\nn)
\end{equation}
induced by $\lambda_{n+1}$ is an equivalence of categories.

Now to prove the proposition it suffices to show that if 
$\lambda_n$ is crystalline, then so is $\lambda_{n+1}$.
Assume that the functor $\lambda_n^*$ is an equivalence.

a) The functor $\lambda_{n+1}^*$ is faithful:
Let $\u M'_{n+1}$ and $\u N'_{n+1}$ be two $\u S'_{n+1}$-windows, let 
$\u M'_i$ and $\u N'_i$ be their reductions
modulo $p^i$, and let $\u M_i$ and $\u N_i$ be the 
images under $\lambda_i^*$.
Assume that $\alpha:\u M'_{n+1}\to\u N'_{n+1}$
is a homomorphism with $\lambda_{n+1}^*(\alpha_{n+1})=0$.
Since the functor $\lambda_n^*$ is faithful we have $\alpha=p^n\beta$ for
a homomorphism $\beta:\u M'_1\to\u N'_1$.
Since the functor $\lambda_1^*$ is faithful we have $\beta=0$, thus $\alpha=0$.

b) The functor $\lambda_{n+1}^*$ is full:
Since a homomorphism $\alpha:M\to N$ can be encoded by the automorphism 
$\left(\begin{smallmatrix}1&0\\\alpha&1\end{smallmatrix}\right)$ of $M\oplus N$,
it suffices to show that $\lambda_{n+1}^*$ is full on isomorphisms.
We use the notation of a).
Let $\alpha_{n+1}:\u M{}_{n+1}\to\u N{}_{n+1}$ be a given isomorphism, and
let $\tilde\alpha_n:\u M'_n\to\u N'_n$ be the unique isomorphism
that lifts the reduction $\alpha_n:\u M\nn\to\u N\nn$.
Since $\tilde\alpha_n$ can be lifted to an isomorphism of pairs
$( M'_{n+1},\Filone  M'_{n+1})\cong( N'_{n+1},\Filone  N'_{n+1})$,
we can assume that $\u M'_n=\u N'_n$
and that $\alpha_n=\id$. Then the equivalence \eqref{Eq-win-lift}
implies that $\alpha$ comes from an isomorphism $\u M'_{n+1}\cong\u N'_{n+1}$.

c) The functor $\lambda_{n+1}^*$ is surjective on isomorphism classes:
This is clear by the equivalence \eqref{Eq-win-lift}.
\end{proof}

\subsection{\texorpdfstring{Descent of windows from $S$ to $\s$}{Descent of windows from S to FrakS}}
\label{Se:descent}

Let $\FS$ be a ring with a fixed element $E\in\FS$ such that
 $(p,E)$ is a regular sequence in $\FS$. We assume that
$\FS$ is complete for the $(p,E)$-adic topology, 
or equivalently that $\FS$ is $p$-adically complete and $\FS_1=\FS/p\FS$
is $E$-adically complete. Then $\FS$ is also $E$-adically complete
without $E$-torsion, and $R=\FS/E\FS$ is $p$-adically complete and $\ZZ_p$-flat.
Let $\varphi:\FS\to\FS$ be a lift of Frobenius. 
As in Example \ref{Ex:frame-FS} we form the frame
$\u\FS=(\FS,\Filone\FS,R,\varphi,\varphi'_1)$.

Let $S$ be the $p$-adic completion of the divided power
envelope of the ideal $E\FS\subseteq\FS$, 
let $\Filone S$ be the kernel of $S\to R$,
and let $\varphi:S\to S$ be the natural extension of $\varphi$.
The ring $S$ is $\Z_p$-flat (see Lemma \ref{Le:Sflat} below)
and is a PD-frame for $R$ in the sense of Definition \ref{Def:PDframe}.
Let
$
\u S=(S,\Filone S,R,\varphi,\varphi_1,p)
$
be the corresponding frame, i.e.\ $\varphi_1(a)=p^{-1}\varphi(a)$ for $a\in\Filone S$. 
With Remarks \ref{Re:Frame} and \ref{Rem:cdetermined} in mind, we assume:
\begin{hypothesis}
The element $c:=\varphi_1(E)$ is a unit of $S$. 
\end{hypothesis}
This guarantees that
the natural homomorphism of rings $\FS\to S$
extends to a $c$-homomorphism of frames $\lambda:\u\FS\to\u S$ in the sense of 
Definition \ref{Def:FrameHom}.
Let $\lambda_n:\u\FS\nn\to\u S\nn$ be its reduction mod $p^n$ as in \S \ref{Se:LiftWin}.

\begin{proposition}
\label{Pr:lambda-crys}
If $p\ge 3$ then the frame homomorphisms\/
$\lambda:\u\FS\to\u S$ 
and\/ $\lambda_n:\u\FS\nn\to\u S\nn$ 
for $n\ge 1$ are all crystalline.
\end{proposition}

This generalises known results.
When $\FS=W(k)[[u]]$ for a perfect field $k$ of characteristic $p$ and when
$E$ is an Eisenstein polynomial so that $R$ is the ring of integers of a $p$-adic local field,
the equivalence between $\u\FS$-windows and $\u S$-windows follows from
\cite[2.2.7, A.6]{KisinFcrystal} when $\varphi(u)=u^p$ and, more directly, 
from \cite[2.2.1]{Caruso-Liu:Quasi} for general $\varphi$.
Using the method of \cite{Caruso-Liu:Quasi}, 
this equivalence is extended in \cite[Prop.~6.3]{Kim:FormallySmooth}
to the case $\FS=R_0[[u]]$ for a $p$-adically complete ring $R_0$ where
$E\in\FS$ is a generalisation of an Eisenstein polynomial.

\begin{proof}
By Proposition \ref{Pr:lift-pn} it suffices to show that the frame homomorphism
$\u\FS_1\to\u S_1$ is crystalline. We prove this using a variant of
\cite[Pr.~2.2.2.1]{Breuil:Construction} and  \cite[Th.~4.1.1]{Breuil:UneApplication}.

Let $\FS_0=\FS/(p,E^p)=\FS_1/(E^p)$. 
For $x\in\FS_1$ we have
$$
\varphi_1'(xE^p)=\varphi(x)\varphi(E)^{p-1}=x^pE^{p(p-1)}
$$
in $\FS_1$.
Thus $\varphi_1'$ preserves the ideal $J=E^p\FS_1$, and (since $p\ge 3$) the restriction
$\varphi_1':J\to J$ is topologically nilpotent.
It follows that (also for $p=2$) we have a frame
$$
\u\FS_0=(\FS_0,\Filone\FS_0,R_1,\varphi,\varphi_1')
$$
such that the projection $\pi:\FS_1\to\FS_0$ is a strict frame
homomorphism $\pi:\u\FS_1\to\u\FS_0$,
and for $p\ge 3$ this frame homomorphism is crystalline by
the general deformation result \cite[Theorem 3.2]{Lau:Frames}.

We can endow the ideal $\Filone\FS_0=E\FS_1/E^p\FS_1$ with the
trivial divided powers $\gamma$ determined by $\gamma_p(E)=0$.
The universal property of $S_1$ gives an extension of
$\pi$ to a divided power homomorphism $\pi':S_1\to\FS_0$, which
maps $\Fil^pS_1$ to zero. 
Now the identity of $\FS_0$ factors into
$$
\FS_0=\FS_1/E^p\FS_1\to S_1/\Fil^pS_1\to\FS_0.
$$
Here the first arrow is easily seen to be surjective, thus both
arrows are bijective. 
In $S$ we have the relation
$$
\varphi_1(E^n/n!)=c^{n}p^{n-1}/n!=\text{unit}\cdot p^{n-1-v_p(n!)},
$$
and $v_p(n!)\le(n-1)/(p-1)$. As $p\ge 3$,
the exponent $n-1-v_p(n!)$ is positive for $n\ge 2$.
Thus $\varphi_1:\Fil^pS_1\to S_1$ is zero, and 
the map $\varphi_1:\Filone S_1\to S_1$ factors over
$\Filone\FS_0\to S_1$ and in particular induces $\bar\varphi_1:\Fil\FS_0\to\FS_0$.
Using that $\lambda$ is a $c$-homomorphism of frames, we find that
$\bar\varphi_1=c\varphi_1'$. Thus $\pi'$ is a $c^{-1}$-homomorphism of frames
$\pi':\u S_1\to\u\FS_0$, which is crystalline by \cite[Theorem 3.2]{Lau:Frames} again.
\end{proof}

\subsection{\texorpdfstring{Dieudonn\'e crystals and the Hodge filtration}{Dieudonn\'e crystals and the Hodge filtration}}\label{DieuHodge}

Let $T$ be a scheme on which $p$ is nilpotent.
We denote by $\Dcat(T)$ the category of Dieudonn\'e crystals over $T$.
Let us recall the definition:

Set $\Sigma:=\Spec \Z_p$, endowed with the structure of a
PD-scheme via the canonical
divided powers on the ideal $p\Z_p$, 
and let $\Cris(T/\Sigma)$ be the big fppf crystalline site as in \cite{BBM}.
If $\scrF$ is a sheaf on $\Cris(T/\Sigma)$ and
$U\to Z$ is any object of $\Cris(T/\Sigma)$, we 
write $\scrF_{Z/U}$ for the corresponding fppf sheaf on $Z$.
In the case $U=Z$ we also write $\scrF_{Z/U}=\scrF_U$.
By definition, an object of $\Dcat(T)$ is a triple $(\scrD,F,V)$ where
$\scrD$ is a crystal of finite locally free $\O_{T/\Sigma}$-modules on $\Cris(T/\Sigma)$
and where $F:\varphi^*\scrD_0\to\scrD_0$ and $V:\scrD_0\to\varphi^*\scrD_0$ are
homomorphisms of crystals satisfying $FV=p$ and $VF=p$; 
here $\scrD_0$ denotes the pullback of $\scrD$ to $T_0=T\times\Spec\F_p$,
and $\varphi$ is the Frobenius endomorphism of $T_0$.
The {\em dual} of $(\scrD,F,V)$ is $(\scrD^t,V^t,F^t)$
for $\scrD^t:=\scrHom_{\O_{T/\Sigma}}(\scrD,\O_{T/\Sigma})$
with $V^t$ and $F^t$ the maps induced 
by precomposition with $V$ and $F$ on $\scrD_0$, respectively.

Let $\pdiv(T)$ denote the category of $p$-divisible groups over $T$.
We have a contravariant functor 
\begin{equation}
\label{Eq:Dfunctor}
\u\D:\pdiv(T)\to\Dcat(T) , \qquad\u\D(G)=(\D(G),F,V)
\end{equation}
where $\D(G)$ is the crystalline Dieudonn\'e module of $G$ defined
in \cite{Mazur-Messing} or \cite{BBM}, and where
$F$ and $V$ are induced by the Frobenius and Verschiebung homomorphism of $G_{T_0}$. 
The functor $\u\D$ is exact by \cite[Pr.~4.3.1 and Th.~3.3.3]{BBM} 
and compatible with duality by  \cite[Sec.~5.3]{BBM}.

We want to take into account the Hodge filtration
of a $p$-divisible group.
Let $(\scrD,F,V)$ be a Dieudonn\'e crystal over $T$.
For each object $(U\to Z)$ of $\Cris(T_0/\Sigma_0)$
we have an exact sequence of locally free $\O_Z$-modules
\begin{equation}
\label{Eq:SeqFV}
\scrD_{Z/U}\xrightarrow{V_{Z/U}}\varphi^*(\scrD_{Z/U})\xrightarrow{F_{Z/U}}
\scrD_{Z/U}\xrightarrow{V_{Z/U}}\varphi^*(\scrD_{Z/U})
\end{equation}
where $\varphi$ denotes the Frobenius endomorphism of $Z$, and
$\im(V_{Z/U})$ and $\im(F_{Z/U})$ are locally free as well.
Indeed, \eqref{Eq:SeqFV} is a complex of locally free modules whose
base change to perfect fields is exact, which implies the assertion.
Moreover, the Frobenius endomorphism $\varphi:Z\to Z$ factors into
$\bar\varphi:Z\to U$. 

\begin{definition}
\label{Def:DF}
We denote by $\DF(T)$ the category of {\em filtered Dieudonn\'e crystals} over $T$
whose objects are quadruples $(\scrD,F,V,\Fil\scrD_T)$
where $(\scrD,F,V)$ is a Dieudonn\'e crystal over $T$ and where 
$\Fil\scrD_T\subseteq\scrD_T$ is an $\O_T$-submodule which is locally a direct summand,
such that for each object $(U\to Z)$ of $
\Cris(T_0/\Sigma_0)$,
we have
\begin{equation}
\label{Eq:DF}
\bar\varphi^*((\Fil\scrD_T)_U)=\im(V_{Z/U})
\end{equation}
inside the locally free $\O_Z$-module $\bar\varphi^*\scrD_U=\varphi^*\scrD_{Z/U}$.
 A short sequence of filtered Dieudonn\'e crystals is called exact if the underlying sequences
of $\O_{T/\Sigma}$-modules $\scrD$ and of $\O_T$-modules $\Filone\scrD_T$ are exact.
The {\em dual} of $(\scrD,F,V,\Filone\scrD_T)$ is defined as $(\scrD^t,V^t,F^t,(\Filone\scrD_T)^\perp)$,
where $(\Filone\scrD_T)^\perp$ is the kernel of the canonical map
$$\scrD^t_T=\scrHom_{\O_T}(\scrD_T,\O_T)\to \scrHom_{\O_T}(\Filone\scrD_T,\O_T).$$
\end{definition}

\begin{remark}
The locally direct summands of $\DDD_T$ that satisfy \eqref{Eq:DF} are the 
filtrations considered in \cite[p.\ 113]{Grothendieck:Montreal}
as part of a so-called {\em admissible} $F$-$V$ crystal.
\end{remark}

For a $p$-divisible group $G$ over $T$ we have a natural exact sequence
of locally free $\O_T$-modules, called the Hodge filtration of $G$,
$$
0\longrightarrow\omega_G\longrightarrow\D(G)_T\longrightarrow\scrLie (G^{\vee})
\longrightarrow 0.
$$

\begin{proposition}
\label{Pr:BT2DF}
The functor\/ $\u\D$ of \eqref{Eq:Dfunctor} extends to a contravariant functor 
$$
\u\D:\pdiv(T)\to\DF(T),\qquad\u\D(G)=(\D(G),F,V,\Filone\D(G)_T)
$$
where $\Fil\D(G)_T$ is the image of $\omega_G$. 
This functor is again exact and compatible with duality.
\end{proposition}

\begin{proof}
We have to verify that \eqref{Eq:DF} holds with this definition of $\Filone\D(G)_T$.
By \cite[Pr.~4.3.10]{BBM} the assertion holds when $U=Z$.
Since $G_U$ can locally in $U$ be lifted to a $p$-divisible group over $Z$,
the assertion follows in general. 
The extended functor $\u\D$ is exact because the functor $G\mapsto\scrLie(G)$
is exact; cf.\ \cite[Th.~3.3.13]{Messing}. 
The extended functor $\u\D$ is compatible with duality by \cite[Pr.~5.3.6]{BBM}.
\end{proof}

\begin{point}
If $T$ is a formal scheme over $\Spf\ZZ_p$ with ideal
of definition $\III\subseteq\O_T$, we define a (filtered) Dieudonn\'e crystal 
over $T$ to be a compatible system of such objects
over the schemes $T_n=V(\III^n)$ for $n\ge 1$. 
The preceding discussion extends to this case.
\end{point}

\subsection{\texorpdfstring{Dieudonn\'e crystals and windows}{Dieudonn\'e crystals and windows}}

Let $\u S=(S,\Filone S,R,\varphi,\varphi_1,p)$ be a PD-frame for a $p$-adically complete ring 
$R$ in the sense of Remark \ref{Re:PDframe}. 
If a Dieudonn\'e crystal $(\scrD,F,V)$ over $\Spf R$ together with a locally direct summand 
$\Filone\scrD_R\subseteq\scrD_R$ is given, one can try to define a window $\u M$ over $\u S$
as follows. Set $M=\scrD_{S}$,
 let $\Filone M\subseteq M$ 
be the inverse image of $\Filone\scrD_{R}\subseteq\scrD_R=M/(\Filone S)M$,
and let  $\Phi:M\to M$ be the $\varphi$-linear map induced by $F$.

\begin{lemma}
\label{Le:DF2Win}
The following conditions are equivalent.
\begin{enumerate}
\item\label{Cond:Win}
We have $\Phi(\Filone M)\subseteq pM$, and $\Phi(M)+p^{-1}\Phi(\Filone M)$ generates $M$ over $S$.
\item\label{Cond:DF}
The equality \eqref{Eq:DF} holds for $U=\Spec R/pR$ and $Z=\Spec S/pS$.
\end{enumerate}
\end{lemma}

Condition \eqref{Cond:Win} means that $\u M=(M,\Filone M,\Phi,\Phi_1=p^{-1}\Phi)$
is a well-defined $\u S$-window. Thus we get:

\begin{proposition}
\label{Pr:DF2Win}
There is an exact functor
\begin{equation*}
\DF(\Spf R)\to\Win(\u S),\qquad\u\scrD\mapsto\u M
\end{equation*}
that is compatible with duality.
\qed
\end{proposition}

\begin{proof}[Proof of Lemma $\ref{Le:DF2Win}$]
We may assume that $p$ is nilpotent in $R$.
Since the Frobenius endomorphism of $S/pS$ induces a homomorphism
$\bar\varphi:R/pR\to S/pS$, the $\varphi$-linear map
$\Phi:M\to M$ induces a $\bar\varphi$-linear map 
$$
\bar\Phi:R/pR\otimes_S M\to M/pM.
$$
The condition $\Phi(\Filone M)\subseteq pM$ is equivalent to
$\bar\Phi(\Filone\scrD_{R/pR})=0$, which translates into
\begin{equation}
\label{Eq:DFincl}
\bar\varphi^*(\Filone\scrD_{R/pR})\subseteq\im(V_{Z/U})
\end{equation}
inside $\bar\varphi^*\DDD_U=\varphi^*\DDD_{Z/U}$
for $U=\Spec R/pR$ and $Z=\Spec S/pS$; this is one inclusion
of the equality \eqref{Eq:DF}.
Assume that \eqref{Eq:DFincl} holds so that we can
define $\Phi_1:=p^{-1}\Phi:\Filone M\to M$.
It remains to show that \eqref{Eq:DFincl} is an equality if and only if
$\Phi(M)+\Phi_1(\Filone M)$ generates $M$ over $S$. 
This is easily verified if $R=k$ is a perfect field and $S=W(k)$.
The general case is reduced to this case as follows.

We choose a normal decomposition $M=L\oplus N$ such that $\Filone M=L\oplus(\Filone S)N$.
The module $M$ is generated by $\Phi(M)+\Phi_1(M)$
if and only if the $\varphi$-linear map of $S$-modules
$\Psi:L\oplus N\to M$ defined by $\Phi_1$ on $L$ and by $\Phi$ on $N$ 
is a $\varphi$-linear isomorphism. 
Every homomorphism $R\to k$ with a perfect field $k$ extends uniquely
to a homomorphism of frames $\u S\to \u W(k)$, so this condition can be checked
over perfect fields.
Similarly, since both sides of \eqref{Eq:DFincl} are direct summands of $M/pM$,
we have equality if and only if equality holds after every base change to a perfect field.
\end{proof}

\subsection{\texorpdfstring{Dieudonn\'e crystals and windows with connection}{Dieudonn\'e crystals and windows with connection}}\label{WinConn}

Assume that $A\to R=A/J$ is a surjective homomorphism of $p$-adically complete rings
where $A$ is $\Z_p$-flat.
Let $S$ be the $p$-adic completion of the divided power envelope of $J$ in $A$ with
respect to $\Sigma$; see \S\ref{DieuHodge}.
Then
$S$ depends only on the ideal $J+pA$. 
Let $\Filone S$ be the kernel of the natural homomorphism $S\to R$. 

\begin{lemma}
\label{Le:Sflat}
If the kernel of $A/pA\to R/pR$ is generated by a regular sequence 
locally in $\Spec A/pA$, then the ring $S$ is $\Z_p$-flat. 
\end{lemma}

This should be standard, but we include a proof for completeness.

\begin{proof}
We may assume that $J_0=\ker(A/pA\to R/pR)$ is generated by a
regular sequence $\bar t_1,\ldots,\bar t_r$. 
Let $t_i\in A$ be an inverse image of $\bar t_i$, and let $J'=(t_1,\ldots,t_r)$ as an ideal of $A$.
Then $t_1,\ldots,t_r$ is a regular sequence, and $R'=A/J'$ is a $\Z_p$-flat $p$-adically complete ring.
Without changing $S$ we can assume that $J=J'$ and $R=R'$.

Let $\Lambda=\Z_p[T_1,\ldots,T_r]$. 
We consider the ring homomorphisms $\Lambda\to A$ defined by $T_i\mapsto t_i$
and $\Lambda\to\Z_p$ defined by $T_i\mapsto 0$. 
Then $\Z_p\otimes_\Lambda A=R$ and $\Tor_i^\Lambda(\Z_p,A)=0$ for $i>0$.
Since $R$ has no $p$-torsion it follows that $\Tor_i^\Lambda(\F_p,A)=0$ for $i>0$.

Let $\Lambda'=\Z_p\left<T_1,\ldots,T_r\right>$ be the divided power
polynomial algebra and let $A'$ be the divided power envelope of $J\subseteq A$
with respect to $\Sigma$.
By the proof of \cite[Pr.~3.4.4]{Berthelot:Crystalline} 
we have $A'=\Lambda'\otimes_\Lambda A$. 
Now $\Lambda'$ is $\Z_p$-flat, and $\Lambda'/p\Lambda'$ is isomorphic to a
direct sum of copies of $\F_p[T_1,\ldots,T_r]/(T_1^p,\ldots,T_r^p)$ as a $\Lambda$-module. 
Since the last module is a finite successive extension of copies of $\F_p$, it follows that
$\Tor_1^\Lambda(\Lambda'/p\Lambda',A)=0$. 
Therefore $A'$ is $\Z_p$-flat, and thus $S$ is $\ZZ_p$-flat.
\end{proof}

From now on we assume that $S$ is $\Z_p$-flat.
Let $\varphi:A\to A$ be a Frobenius lift. 
It extends naturally to an endomorphism $\varphi: S\to S$ 
that induces the Frobenius on $ S/pS$ and makes $S$ into
a PD-frame for $R$ in the sense of Definition \ref{Def:PDframe}.
Let 
$$
\u S=(S,\Filone S,\varphi,\varphi_1,p).
$$
be the associated frame as in Remark \ref{Re:PDframe}, i.e.\ $\varphi_1=p^{-1}\varphi$ on $\Filone S$.

Assume that  $A_0=A/pA$ has a finite $p$-basis $(x_i)$.
Then $\wh\Omega_A$ is a free $A$-module with basis $(dx_i)$.
The universal derivation $d:A\to\wh\Omega_{A}$ extends to a continuous connection 
\begin{equation}
	\nabla_S:S\rightarrow S\otimes_{A} \wh{\Omega}_{A}
\quad\text{satisfying}\quad
	\nabla_{S}(j^{[n]})=j^{[n-1]} \otimes dj\label{Sconn}
\end{equation}	
for any $j\in J$; see Remarks 2.2.4 (d) and 2.2.1.3 of \cite{deJong:Crystalline}.

\begin{definition}
\label{Def:WinConn}
Let $\Win(\u S)^\nabla$ be the category of pairs $(\u M,\nabla)$ where
$\u M$ is an $\u S$-window and where $\nabla:M\to M\otimes_A\wh\Omega_A$
is a connection over $\nabla_S$ for which $\Phi$ is horizontal, 
i.e.\ $\nabla\circ\Phi=(\Phi\otimes d\varphi)\circ\nabla$.
Morphisms in $\Win(\u S)^\nabla$ are homomorphisms of $\u S$-windows that
are compatible with the connections.
\end{definition}

This definition is justified by the following lemma; cf.\ \cite[Def.\ 3.2.3]{Kim:FormallySmooth}.

\begin{lemma}
\label{Le:WinConn}
For $(\u M,\nabla)\in\Win(\u S)^\nabla$, the connection $\nabla$ is integrable
and topologically quasi-nilpotent.
\end{lemma}

\begin{proof}
The $\varphi$-linear map $d\varphi:\wh\Omega_A\to\wh\Omega_A$ is divisible by $p$
since $\varphi(x)\equiv x^p$ mod $pA$ for $x\in A$.
The obvious homomorphism $\nabla^2:M\to M\otimes_A\bigwedge^2\wh\Omega_A$ 
is $S$-linear
and satisfies $\nabla^2\circ\Phi=(\Phi\otimes\bigwedge^2(d\varphi))\circ\nabla^2$.
Since $p$ is not a zero divisor in $S$, it follows that 
$\nabla^2\circ\Phi_1=p(\Phi\otimes\bigwedge^2(p^{-1}d\varphi))\circ\nabla^2$ on $\Filone M$.
Since the image of $\Phi_1$ generates $M$ as an $S$-module,
we deduce that if $\nabla^2$ is divisible by $p^n$ for some $n\ge 0$, 
then $\nabla^2$ is divisible by $p^{n+1}$.
Thus $\nabla^2=0$, i.e.\ $\nabla$ is integrable.

The proof that $\nabla$ is topologically quasi-nilpotent follows 
\cite[2.4.8]{deJong:Crystalline}.
Writing $j:M\to \varphi^*M$ for the $\varphi$-linear map $j(x):=1\otimes x$,
there is a well-defined connection 
$\varphi^*\nabla:\varphi^*M\to \varphi^*M\otimes_A\wh\Omega_A$ over
$\nabla_S$ such that $\varphi^*\nabla(j(x))=(j\otimes d\varphi)(\nabla(x))$, 
and the homomorphisms
$F:\varphi^*M\to M$ and $V:M\to \varphi^*M$ are horizontal with respect to 
$\nabla$ and $\varphi^*\nabla$. 
Let $\vartheta:\wh\Omega_A\to S$ be an $A$-linear map such that the associated
derivation of $S/pS$ is nilpotent, and consider the associated
differential operators $D=\vartheta\circ\nabla:M\to M$
and $D'=\vartheta\circ(\varphi^*\nabla):\varphi^*M\to \varphi^*M$. 
Since $d\varphi$ is divisible by $p$ we have $D'(1\otimes x)\in p\varphi^*M$.
Thus $D'$ is nilpotent 
on $\varphi^*(M/pM)$.
Using the exact sequence with horizontal maps
$$
\varphi^*(M/pM)\xrightarrow{F}M/pM\xrightarrow{V}\varphi^*(M/pM)
$$
it follows that $D$ is nilpotent 
on $M/pM$, thus $\nabla$ is topologically quasi-nilpotent.
\end{proof}

\begin{proposition}
\label{Pr:DF2WinConn}
	There is an exact equivalence of categories
	\begin{equation*}
		\DF(\Spf R)
		\to \Win({\u{S}})^{\nabla}
	\end{equation*}
	that is compatible with duality.
\end{proposition}

This is similar to \cite[Pr.\ 3.2.5]{Kim:FormallySmooth}
(but there condition \eqref{Eq:DF} is left out).

\begin{proof}
The functor in the proposition exists and is fully faithful by Proposition \ref{Pr:DF2Win} 
together with the equivalence between crystals in finite locally free 
$\O_{\Spf R/\Sigma}$-modules over $\Cris(\Spf R/\Sigma)$ 
and finite projective $S$-modules with an integrable topologically quasi-nilpotent
connection; see \cite[Pr.~2.2.2]{deJong:Crystalline}. 
It remains to verify that the functor is essentially surjective. 
Using Lemmas \ref{Le:DF2Win} and \ref{Le:WinConn}, 
this translates into the following assertion: 
If $(\scrD,F,V)$ is a Dieudonn\'e crystal over $\Spf R$
and if $\Filone\scrD_R\subseteq\scrD_R$ is a locally direct
summand such that condition \eqref{Eq:DF} holds for 
$U=\Spec R/pR$ and $Z=\Spec S/pS$, then \eqref{Eq:DF} holds for all $(U\to Z)$. 
We may assume that $U=\Spec B$ and $Z=\Spec C$ are affine.
We claim that the given morphism $U\to\Spec R/pR\to\Spec A/pA$ extends to $Z$.
Consider the following diagram where $\alpha$ is the given morphism and $\phi$ is
the Frobenius, and the existence of $\tilde\alpha$ has to be shown.
\[
\xymatrix@M+0.2em{
C \ar[d]_\pi &
A/pA \ar@{-->}[l]_\beta \ar[d]^\phi \\
B &
A/pA \ar[l]_\alpha \ar@{-->}[ul]_{\tilde\alpha}
}
\]
For $a\in A/pA$ choose $c\in C$ with $\pi(c)=\alpha(a)$ and put $\beta(a)=c^p$.
This does not depend on the choice of $c$ because $\pi$ is a PD extension 
of $\FF_p$-algebras, which implies that each $x\in\Ker(\pi)$ satisfies $x^p=0$. 
It follows that $\beta$ is a ring homomorphism such that the outer square commutes. 
Now, if $x_i$ form a $p$-basis of $A/pA$ and if $y_i\in C$ are chosen with 
$\pi(y_i)=\alpha(x_i)$, then there is a unique extension $\tilde\alpha$ of $\beta$
with $\tilde\alpha(x_i)=y_i$. This proves the claim.
From the universal mapping property of divided power envelopes, 
we then obtain a morphism of PD-thickenings from $(U\to Z)$ to 
$(\Spec R/pR\to\Spec S/pS)$, and the assertion follows by pull back. 
\end{proof}

\begin{remark}
The preceding proof shows that for the specific
PD-frames considered here,
the equivalent conditions of Lemma \ref{Le:DF2Win} imply
the equality \eqref{Eq:DF} for all $(U\to Z)$. This does not hold for
arbitrary PD-frames.
\end{remark}

\subsection{Dieudonn\'e theory over complete regular local rings}
\label{Subse:DieuReg}

Let $R$ be a complete regular local ring with perfect residue field $k$
of characterstic $p\ge 3$ and with fraction field of characteristic zero.
We put $\FS=W(k)[[u_1,\ldots,u_d]]$ where $d=\dim(R)$
and choose a homomorphism $\pi:\FS\to R$ such that the
elements $\pi(u_i)$ generate the maximal ideal of $R$.
Then it follows from the Cohen structure theorem for complete regular local
rings \cite[Theorem 15, Corollary 3]{Cohen}, that 
$R=\FS/E\FS$ for a power series $E\in\FS$ with constant
term of $p$-value one; see \cite[\S 7]{Lau:Frames}.

As in \S \ref{Se:descent} let $S$ be the $p$-adic completion
of the divided power envelope of the ideal $E\FS\subset\FS$.
Let $\varphi:\FS\to\FS$ be a lift of Frobenius and denote its
extension to $S$ by $\varphi$ again. 
The element $c=\varphi(E)/p$ is a unit of $S$; see \cite[Le.\ 6.1]{Lau:Dieudonne}.
We obtain a frame homomorphism $\lambda:\u\FS\to\u S$,
which is crystalline by Proposition \ref{Pr:lambda-crys}.

Since $\varphi$ is a lift of Frobenius, its derivative
$d\varphi:\wh\Omega_\FS\to\wh\Omega_\FS$ is divisible by $p$.
Let $(d\varphi)_1:=p^{-1}d\varphi$ as an endomorphism of $\wh\Omega_\FS$.

\begin{proposition}
\label{Pr:BT2Win}
If $(d\varphi)_1$ is nilpotent on $\wh\Omega_\FS\otimes_\FS k$, 
then the contravariant
functor $\pdiv(R)\to\Win(\u S)$ given by the Dieudonn\'e crystal 
$($Propositions $\ref{Pr:BT2DF}$ and $\ref{Pr:DF2Win}$$)$
is an equivalence of categories.
\end{proposition}

\begin{remarks}
The hypothesis of Proposition \ref{Pr:BT2Win} holds when $\varphi(u_i)=u_i^p$ because then
$(d\varphi)_1$ is zero on $\wh\Omega_\FS\otimes_\FS k$.
In the case $d=1$ with $\varphi(u_i)=u_i^p$, 
the proposition is due to Breuil \cite[Th.~4.2.2.9]{Breuil:Groupes};
alternative proofs are given by Kisin \cite[Prop.~A6]{KisinFcrystal} and by Zink 
\cite[Prop.~3.7]{Zink:Windows}.
The general case does not follow from \cite[Thm.\ 6]{Zink:Windows}
because the ring $S$ is a $\hat {\mathcal{Z}}$-ring
in the sense of \cite{Zink:Windows} only when $d=1$.
We emphasize that even in the case $d=1$, Proposition \ref{Pr:BT2Win}
provides a more general result than is available in 
\cite{Breuil:Groupes}, \cite{KisinFcrystal}, or \cite{Zink:Windows}
(which all work with $\varphi(u_1):=u_1^p$)
due to the greater flexibility in the choice of Frobenius lift.
\end{remarks}

\begin{proof}[Proof of Proposition $\ref{Pr:BT2Win}$]
The proposition is a consequence of the results of \cite{Lau:Dieudonne},
with a possible variation using Proposition \ref{Pr:lambda-crys}.
Let $\hat W(R)=\WW(R)$ be the subring of $W(R)$ defined in \cite{Zink:Dieudonne},
and let $\u\WW(R)=(\WW(R),\II_R,R,\sigma,\sigma_1)$ be the associated frame where
$\sigma_1$ is the inverse of the Verschiebung; see \cite[Sec.~2.3]{Lau:Dieudonne}.
This is a PD-frame because $\WW(R)$ is $p$-adically complete with no $p$-torsion 
by \cite[Prop.~1.14]{Lau:Dieudonne}, and the ideal $\II_R$ 
carries divided powers by \cite[Lemma 1.16]{Lau:Dieudonne}.

We have natural frame homomorphisms
\begin{equation}
\label{Eq:BT2Win}
\u\FS\xrightarrow{\;\;\lambda\;\;}\u S\xrightarrow{\;\;\tilde\varkappa\;\;}\u\WW(R)
\end{equation}
where $\lambda$ is the inclusion. 
Namely, there is a natural ring homomorphism $\varkappa:\FS\to W(R)$ compatible with
the projections to $R$ such that $\sigma\varkappa=\varkappa\varphi$, 
and the image of $\varkappa$ lies in the subring  $\hat W(R)$ if and only if 
$(d\varphi)_1$ is nilpotent on $\wh\Omega_\FS\otimes_\FS k$;
see \cite[Prop.~6.2]{Lau:Dieudonne}
(by a change of variables one can assume that $\varphi$
preserves the ideal of $\FS$ generated by the $u_i$, which is assumed in loc.cit.). 
Since $\u\WW(R)$ is a PD-frame, 
$\varkappa$ extends to a ring homomorphism $\tilde\varkappa:S\to\WW(R)$,
which is a strict frame homomorphism as in \eqref{Eq:BT2Win}.

Now $\tilde\varkappa$ is crystalline by \cite[Thm.\ 7.2]{Lau:Dieudonne};
alternatively one can use that $\lambda$ is crystalline by Proposition \ref{Pr:lambda-crys}, 
and that $\varkappa=\tilde\varkappa\circ\lambda$ is crystalline by \cite[Thm.~6.5]{Lau:Dieudonne}.
Thus the proposition holds if and only if the composition
\[
\pdiv(R)\longrightarrow\Win(\u S)\xrightarrow{\;\;\tilde\varkappa^*\;\;}\Win(\u\WW(R))
\]
is an equivalence. This composition is given by evaluating the crystalline
Dieudonn\'e module as in Proposition \ref{Pr:DF2Win}, and the claim follows
from the crystalline version of the equivalence between $p$-divisible groups and
Dieudonn\'e displays in \cite[Cor.\ 5.4]{Lau:Dieudonne}.
\end{proof}

\begin{proposition}
\label{Pr:WinConn}
If $(d\varphi)_1$ is nilpotent on $\wh\Omega_\FS\otimes_\FS k$, 
then the forgetful functor $\Win(\u S)^\nabla\to\Win(\u S)$ is an equivalence
of categories.
\end{proposition}

\begin{proof}
Let $\u M$ be an $\u S$-window. We have to show that there is a unique connection
$\nabla:M\to M\otimes_\FS\wh\Omega$ over $\nabla_S$
such that $\nabla\circ\Phi=(\Phi\otimes d\varphi)\circ\nabla$.
As earlier we define $F:\varphi^*M\to M$ and $V:M\to \varphi^*M$ by
$F(1\otimes x)=\Phi(x)$ for $x\in M$ and $V(\Phi_1(y))=1\otimes y$ for $y\in\Filone M$;
see Remarks \ref{Re:Window}.

We choose a basis of $M$ and denote by $A$ and $B$ the corresponding matrices 
of $F$ and $V$. We may assume that the basis of $M$ lies in the image of $\Phi_1$.
Then $B=\varphi(\tilde B)$ for another matrix $\tilde B$. Under the identification
$M\cong S^n$ we have $\nabla(x)=\nabla_S (x)+C\cdot x$ for $x\in S^n$, 
where $C$ is a matrix with coefficients in $S\otimes_\FS\wh\Omega_\FS$. 
The condition on $\nabla$ translates into the equation
\[
\nabla_S(A)+CA=A\cdot d\varphi(C)
\]
(here $d\varphi$ is short for $\varphi\otimes d\varphi$), or equivalently (after right multiplying by $B$)
\[
\nabla_S(A)\cdot B+pC=pA\cdot(d\varphi)_1(C)\cdot B.
\]
Now $AB=A\varphi(\tilde B)=pI_n$ has zero derivative, thus
\[
\nabla_S(A)\cdot B=-pA\cdot(d\varphi)_1(\nabla_S(\tilde B)),
\]
and the condition becomes (after dividing by $p$)
\begin{equation}
\label{Eq:U-1}
C-\UUU(C)=D
\end{equation}
with $\UUU(C)=A\cdot(d\varphi)_1(C)\cdot B$ and $D=A\cdot(d\varphi)_1(\nabla_S(\tilde B))$.
Here $\UUU$ is a $\varphi$-linear endomorphism of the free $S$-module 
of $n\times n$ matrices over $S\otimes_{\FS}\wh\Omega_\FS$,
and our goal is to  solve the equation $(1-\UUU)(C)=D$ for $C$.
Set $y^{(n)}:=\varphi^n(y)$. From the 
definition of $\UUU$, we compute that
the $n$-th iterate of $C\mapsto \UUU(C)$ is given by
\[
C\mapsto AA^{(1)}\cdots A^{(n-1)}(d\varphi)_1^n(C)(B^{(n-1)}\cdots B^{(1)}B).
\]
Now we are facing the difficulty that the inclusion map $\FS\to S$ is continuous with respect to
the $\Fm_\FS$-adic topology on $\FS$ and the $p$-adic topology on $S$
only when $d=1$.
In that case, the hypothesis on $(d\varphi)_1$ implies that
for large $n$, the semilinear endomorphism  
$(d\varphi)_1^n$ of $S\otimes_\FS\wh\Omega_\FS$ is divisible by $p$.
It follows that $1-\UUU$ is invertible,
so there is a unique solution $C$
to (\ref{Eq:U-1}) as desired.

For $d\ge2$ we can argue as follows. Let $T=\FS[[E^p/p]]\subseteq S$. 
By Proposition \ref{Pr:lambda-crys} we may assume that $\u M$ comes from an $\u\FS$-window
$\u M'$. If the basis of $M$ is chosen in $p^{-1}\Phi(M')$, then $A$ and $B$ have
coefficients in $T$. Moreover, since $\varphi(E)\in pT$ it is easy to see that
$\varphi:S\to S$ factors over $\bar\varphi:S\to T$, and $(d\varphi_1)$ induces 
a $\bar\varphi$-linear map $S\otimes_\FS\wh\Omega_{\FS}\to T\otimes_{\FS}\wh\Omega_\FS$.
Thus the solutions of \eqref{Eq:U-1} over $S$ and over $T$ are the same.
The hypothesis on $(d\varphi)_1$ implies that the semilinear endomorphism
$(d\varphi)_1^n$ of $T\otimes_\FS\wh\Omega_\FS$ is nilpotent modulo the
maximal ideal of $T$, so again there is a unique solution $C$ to $(\ref{Eq:U-1})$.
\end{proof}

\begin{corollary}
\label{Co:BT2DF}
The filtered crystalline Dieudonn\'e module functor of Proposition 
\ref{Pr:BT2DF} 
induces a contravariant equivalence of categories
\[
\pdiv(R)\xrightarrow\sim\DF(\Spf R).
\] 
\end{corollary}

\begin{proof}
Chose $\varphi:\FS\to\FS$ such that $(d\varphi)_1$ is nilpotent on $\wh\Omega_\FS\otimes_\FS k$,
for example $\varphi(u_i)=u_i^p$. We have a sequence of functors
\begin{equation}
\label{Eq:seq-funct}
\pdiv(R)\to\DF(\Spf R)\to 
\Win(\u S)^{\nabla}\to 
\Win(\u S)
\end{equation}
where the second and third arrow are equivalences by Propositions \ref{Pr:DF2WinConn}
and \ref{Pr:WinConn}, and the total composition is an equivalence by Proposition \ref{Pr:BT2Win}.
Thus the first arrow is an equivalence.
\end{proof}

\begin{remark}
In conclusion, we see that in \eqref{Eq:seq-funct}, the first
arrow is an equivalence, and the second arrow
is an equivalence for every choice of Frobenius lift $\varphi$
on $\FS$, while the third arrow is an equivalence when $(d\varphi)_1$ is nilpotent
on $\wh\Omega_\FS\otimes_\FS k$.
\end{remark}

\section{\texorpdfstring{Dieudonn\'e theory over cyclotomic rings}
{Dieudonn\'e theory over cyclotomic rings}}
\label{Se:DieuCyc}

In this section we will prove Theorem \ref{Th:Main}.
We begin by defining the setup and the relevant module categories.

\subsection{Strict actions of the cyclotomic Galois group}

Fix a perfect field $k$ of characteristic $p>2$, and let $F=W(k)\otimes\QQ$
and $F_s=F(\mu_{p^s})$ for $s\ge 0$ and $F_\infty=\bigcup_sF_s$. 
We choose a compatible system of primitive $p^s$-th roots of unity $\varepsilon^{(s)}\in F_s$.
Fix $d\ge 0$.

For $s\ge 0$ we consider the surjective homomorphism of $W(k)$-algebras
\[
\FS:=W(k)[[u,t_1,\ldots,t_d]]\xrightarrow{\pi_s} R_s:=\OOO_{F_s}[[t_1,\ldots,t_d]]
\]
defined by $1+u\mapsto \varepsilon^{(s)}$ and $t_i\mapsto t_i$.
We equip $\FS$ with the extension $\varphi:\FS\to\FS$ of the Frobenius of $W(k)$
determined by 
$\varphi(1+u)=(1+u)^p$ and $\varphi(t_i)=t_i^p$,
and observe that the kernel of $\pi_s$ is generated by the element $E_s$ defined by $E_0=u$ and 
$E_s=\varphi^s(u)/\varphi^{s-1}(u)$ for $s\ge 1$.
Note that we have $(d\varphi)_1(du)=(1+u)^{p-1}du$, 
so that 
this choice of $\varphi$ {\em does  not} satisfy the hypotheses
of Propositions \ref{Pr:WinConn} and \ref{Pr:BT2Win}.
In order to prove analogues of these results in the present ``cyclotomic"
setting, we will therefore need to introduce the auxiliary structure
of an action of $\Gal(F_{\infty}/F)$.

Let $\Gamma_s=\Gal(F_\infty/F_s)$ and let $\chi:\Gamma_0\to\ZZ_p^*$
be the $p$-adic cyclotomic character. 
We let $\Gamma_0$ act on the $W(k)$-algebra $\FS$ by $\gamma(1+u)=(1+u)^{\chi(\gamma)}$
and $\gamma(t_i)=t_i$. Then $\pi_s$ is equivariant with respect to the natural
action of $\Gamma_0$ on $R_s$ that fixes all $t_i$. 
The action of $\Gamma_0$ on $\FS$ commutes with $\varphi$,
and it is continuous in the following sense.

\begin{lemma}
\label{Le:GammaFS}
Let $s\ge 0$.
The action of\/ $\Gamma_0$ on $\FS$ preserves the ideal $\varphi^s(u)\FS$, 
and the induced action of\/ $\Gamma_s$ on $\FS/\varphi^s(u)\FS$ is trivial.
\end{lemma}

\begin{proof}
We have $\varphi^s(u)=E_0E_1\cdots E_s$ as a product of pairwise coprime prime elements of 
the factorial ring $\FS$,
and for $0\le i\le s$ the group $\Gamma_0$ acts on $R_i=\FS/E_i\FS$ with trivial action of $\Gamma_s$.
The lemma follows easily.
\end{proof}

For $r\ge 1$ we consider the frame $\u\FS\rr=(\FS_r,E_r\FS_r,R_r,\varphi,\varphi'_1)$
with $\FS_r=\FS$ and $\varphi'_1(E_rx)=\varphi(x)$. Here $\varphi_1'$ depends on $r$.
Each $\gamma\in\Gamma_0$ defines a window automorphism $\gamma:\u\FS\rr\to\u\FS\rr$, 
which is a $c_\gamma$-automorphism with $c_\gamma=\gamma(\varphi(E_r))/\varphi(E_r)$, 
i.e.\ we have $\varphi_1'\gamma=c_\gamma\cdot\gamma\varphi'_1$.
Let $S_r$ be the $p$-adic completion of the divided power envelope of the ideal $E_r\FS$ of $\FS$.
We obtain a frame $\u S\rr=(S_r,\Filone S_r, R_r, \varphi,\varphi_1)$ with $\varphi_1=p^{-1}\varphi$.
The inclusion $\u\FS\rr\to\u S\rr$ is a $c$-homomorphism of frames with $c=\varphi(E_r)/p$.
The action of $\Gamma_0$ on $\FS$ induces an action on $\u S\rr$ by strict frame automorphisms.

\begin{remark}
\label{Re:System}
It will be convenient to denote the variable $u$ of $\FS_r$ by $u_r$ 
and to consider $\FS_r$ as a subring of $\FS_{r+1}$ by letting $u_r=\varphi(u_{r+1})$. 
This makes $\u\FS\rr$ into a subframe of $\u\FS{}_{r+1}$ 
extending the obvious inclusion $R_r\to R_{r+1}$.
The elements $u_0=\varphi^{r}(u_r)$ and $E=E_r(u_r)=u_0/u_1$ and $c_\gamma$
are independent of $r$; moreover $c_\gamma$ lies in $W(k)[[u_0]]$.
Similarly, $\u S\rr$ becomes a subframe of $\u S{}_{r+1}$, and $c$ is independent of $r$.
\end{remark}

In the following let subgroups $\Gamma_r\subseteq\Gamma'\subseteq\Gamma\subseteq\Gamma_0$ 
be given with $r\ge 1$.
We note that the projection $\FS\to\FS/u\FS$ maps $E_r$ into the divided power ideal $(p)$ and thus
extends to a ring homomorphism $S_r\to\FS/u\FS$.

\begin{definition}
\label{Def:GammaWin}
An action of $\Gamma$ on an $\u\FS\rr$-window or $\u S\rr$-window $\u M$ 
by window automorphisms over the given action of $\Gamma$ on $\u\FS\rr$ 
or $\u S\rr$ is called \emph{$\Gamma'$-strict} if the induced action of $\Gamma'$ 
on $M\otimes_\FS\FS/u\FS$ or $M\otimes_{S_r}\FS/u\FS$ is trivial.
We denote by $\Win(\u\FS\rr)^{\Gamma}_{\Gamma'}$ and $\Win(\u S\rr)^{\Gamma}_{\Gamma'}$ 
the categories of windows with a $\Gamma'$-strict action of $\Gamma$.
Definition \ref{Def:WinConn} with $\FS$ in place of $A$ gives a category $\Win(\u S\rr)^\nabla$.
Let $\Win(\u S\rr)^{\nabla,\Gamma}_{\Gamma'}$ be the category of objects $(\u{M},\nabla)$ of 
$\Win(\u S\rr)^\nabla$ that are equipped with a $\Gamma'$-strict action of $\Gamma$
which is horizontal in the sense that for all $\gamma\in \Gamma$, one has
$\nabla\circ \gamma = (\gamma\otimes d\gamma)\circ\nabla$.

For simplicity we write 
$\Win(\u\FS\rr)_\Gamma=\Win(\u\FS\rr)_\Gamma^\Gamma$ etc.\ 
for the category of windows with a strict action of $\Gamma$,
and $\Win(\u\FS\rr)^\Gamma=\Win(\u\FS\rr)^\Gamma_{\Gamma_r}$ etc.\
for the category of windows with a $\Gamma_r$-strict action of $\Gamma$.
\end{definition}

\begin{proposition}
\label{Pr:WinGammaFS-S}
The base change functor 
$\Win(\u\FS\rr)^{\Gamma}_{\Gamma'}\to\Win(\u S\rr)^{\Gamma}_{\Gamma'}$
is an exact equivalence of categories.
\end{proposition}

\begin{proof}
This is a consequence of Proposition \ref{Pr:lambda-crys},
using that an action of $\Gamma$ on an $\u\FS\rr$-window or $\u S\rr$-window $\u M$ 
can be given by window isomorphisms $\gamma^*(\u M)\to\u M$
for $\gamma\in\Gamma$; see \eqref{bcfunc}. 
The strictness condition is preserved when passing from $\u\FS\rr$ to $\u S\rr$.
\end{proof}

\begin{definition}
\label{Def:KisinGamma}
An action of $\Gamma$ on a BT 
module $(\FM,\varphi)$ over
$\u\FS\rr$ by semilinear automorphisms is called \emph{$\Gamma'$-strict} if the induced
action of $\Gamma'$ on $\FM/u\FM$ is trivial.
The category of BT 
modules over $\u\FS\rr$ with a $\Gamma'$-strict action of $\Gamma$ 
is denoted by $\BT(\u\FS\rr)^{\Gamma}_{\Gamma'}$. Again we write
$\BT(\u\FS\rr)_\Gamma=\BT(\u\FS\rr)^\Gamma_\Gamma$ and 
$\BT(\u\FS\rr)^\Gamma=\BT(\u\FS\rr)^\Gamma_{\Gamma_r}$.
\end{definition}

\begin{lemma}
\label{Le:WinKisinGamma}
The equivalence of Lemma $\ref{Le:WinKisin}$ extends to an exact equivalence
$$
\Win(\u\FS\rr)^{\Gamma}_{\Gamma'}\to\BT(\u\FS\rr)^{\Gamma}_{\Gamma'},
\qquad\u M\mapsto\u\FM
$$ 
where $\gamma\in\Gamma$ acts on $\FM=\Filone M$ by $E/\gamma(E)\cdot(\gamma_M|_{\Filone M})$.
\end{lemma}

\begin{proof}
Assume that $\u M\in\Win(\u\FS\rr)$ and $\u\FM\in\BT(\u\FS\rr)$ correspond to each other,
in particular $\FM=\Filone M$ and $M=\varphi^*\FM$.
If $(\FM,\varphi)$ carries an action of $\Gamma$, we let $\gamma\in\Gamma$ act
on $M$ by $\gamma_\FS\otimes\gamma_\FM$. 
One verifies that this definition and the construction of the lemma give 
well-defined and mutually inverse functors between $\Win(\u\FS\rr)^{\Gamma}_{\Gamma'}$
and $\BT(\u\FS\rr)_{\Gamma'}^{\Gamma}$.
\end{proof}

\begin{remark}
\label{Re:strict}
In the definition of strict actions, we could pass from $\FS/u\FS$ to $W(k)=\FS/(u,t_1,\ldots,t_d)$ 
without changing the resulting category; of course this is relevant only when $d\ge 1$.
More precisely,
in Definition \ref{Def:KisinGamma} let $\FM_0=\FM/u\FM$ and let $\bar\FM_0=\FM\otimes_\FS W(k)$.
Then $\Gamma'$ acts trivially on $\FM_0$ if and only it acts trivialy $\bar\FM_0$.

This follows from the fact that the reduction map 
$\End(\FM_0,\varphi)\to\End(\bar\FM_0,\varphi)$ is injective.
Indeed, let $J=(t_1,\ldots,t_d)\FS/u\FS$, choose a basis of $\FM_0$, and let $A$ be the
corresponding matrix over $\FS/u\FS$ of $\varphi$. Then $pA^{-1}$ exists in $\FS/u\FS$.
An endomorphism of $(\FM_0,\varphi)$ is given by a matrix $C$ that satisfies
$CA=A\varphi(C)$, or equivalently $pC=A\varphi(C)pA^{-1}$. 
Assume that $C$ has coefficients in $J$.
Since $\varphi(J)\subseteq J^{p}$ and since each $(\FS/u\FS)/J^m$ is torsion free,
it follows that the coefficients of $C$ lie in $J^m$ for all $m\ge 1$, and thus $C=0$.

Similarly, in Definition \ref{Def:GammaWin} let $M_0=M\otimes_{\FS}\FS/u\FS$ 
or $M_0=M\otimes_{S_r}\FS/u\FS$,
and let $\bar M_0=M_0\otimes_{\FS/u\FS}W(k)$.
Then $\Gamma'$ acts trivially on $M_0$ if and only if it acts trivially on $\bar M_0$;
this follows by the proofs of Lemma \ref{Le:WinKisinGamma} and Proposition \ref{Pr:WinGammaFS-S}.
\end{remark}

\begin{remark}
\label{Re:Dual}
The duality of windows and BT modules extends naturally to a duality of such
objects equipped with a $\Gamma'$-strict action of $\Gamma$.
For $\u M\in\Win(\u S\rr)^{\Gamma}_{\Gamma'}$ the associated action on $\u M^t$
is simply the contragredient action
defined by $\gamma(f)(\gamma(m))=f(m)$ for $m\in M$ and $f\in M^t$.
Over $\FS\rr$ a twist occurs: For $\gamma\in\Gamma_0$, the infinite product
\begin{equation}
\label{Eq:lambda_gamma}
\lambda_\gamma=\prod_{n\ge 0}\varphi^n(E/\gamma(E))
\end{equation}
converges in $\FS_r$ and is independent of $r$. 
For $\u\FM\in\BT(\u\FS\rr)^{\Gamma}_{\Gamma'}$ we define an action on $\u\FM^t$ 
by $\gamma(f)(\gamma(m))=\lambda_\gamma\cdot\gamma(f(m))$,
while for $\u M\in\Win(\u\FS\rr)^\Gamma_{\Gamma'}$ we define an action on $\u M^t$ 
by $\gamma(f)(\gamma(m))=\varphi(\lambda_\gamma)\cdot\gamma(f(m))$.
One can verify that these definitions give $\Gamma'$-strict actions of $\Gamma$ and that 
the equivalences of Lemma \ref{Le:WinKisinGamma} and Proposition \ref{Pr:WinGammaFS-S}
preserve duality.
\end{remark}

When $d\ge 1$ we also need the following variant:

\begin{definition}
Let $\Omega_0:=\wh\Omega_{\FS/W(k)[[t_1,\ldots,t_d]]}=\FS\,du$,
and let $\nabla_{S_r,0}:S_r\to S_r\otimes_\FS\Omega_0$ be the composition of
$\nabla_{S_r}$ and the natural map $\wh\Omega_\FS\to\Omega_0$.
We denote by $\Win(\u{S}\rr)^{\nabla_0}$ the category of $\u{S}\rr$-windows $\u M$ equipped 
with a connection $\nabla_0:M\to M\otimes_\FS\Omega_0$ over $\nabla_{S_r,0}$ 
with respect to which $\Phi$ is horizontal.
Let $\Win(\u{S}\rr)^{\nabla_0,\Gamma}_{\Gamma'}$ be the category of objects of
$\Win(\u{S}\rr)^{\nabla_0}$ with a $\Gamma'$-strict and horizontal action of $\Gamma$.
\end{definition}

\begin{lemma}
\label{Le:nabla0}
The functors
$$
\Win(\u{S}\rr)^\nabla\to\Win(\u{S}\rr)^{\nabla_0}
\qquad\text{and}\qquad
\Win(\u{S}\rr)^{\nabla,\Gamma}_{\Gamma'}\to\Win(\u{S}\rr)^{\nabla_0,\Gamma}_{\Gamma'}
$$
defined by composing a connection $M\to M\otimes\wh\Omega_\FS$ with the
natural map $\wh\Omega_\FS\to\Omega_0$ are equivalences.
\end{lemma}

\begin{proof}
We have $\wh\Omega_\FS=\Omega_0\oplus\Omega_1$ with
$\Omega_1=\wh\Omega_{\FS/W(k)[[u]]}$. The endomorphism
$(d\varphi)_1=p^{-1}d\varphi$ of $\wh\Omega_\FS$ preserves this decomposition, 
and its restriction to $\Omega_1$
is topologically nilpotent due to the choice $\varphi(t_i)=t_i^p$. 
The lemma now follows from the proof of Proposition \ref{Pr:WinConn}.
\end{proof}

\subsection{Construction of strict actions}

Next we study how $p$-divisible groups, or equivalently filtered Dieudonn\'e crystals
over $R_r$ are related to windows with strict actions of $\Gamma_r$.
Recall that $r\ge 1$.

\begin{lemma}
\label{Le:DF2WinConnGamma}
The evaluation functor of Proposition $\ref{Pr:DF2WinConn}$ extends to an 
exact functor
\[
\DF(\Spf R_r)\to\Win(\u{S}_r)^{\nabla,\Gamma_r}
\]
that preserves duality.
\end{lemma}

\begin{proof}
Let $\u\DDD\in\DF(\Spf(R_r))$ map to $(\u M,\nabla)\in\Win(\u S_r)^\nabla$, so in particular
$M=\DDD_{S_r/R_r}$.
Since each $\gamma\in\Gamma_r$ defines an endomorphism of the PD-extension $S_r\to R_r$
that commutes with $\varphi$ and is trivial on $R_r$, we obtain a horizontal action of $\Gamma_r$
on $\u M$. The action is strict because $M\otimes_{S_r}\FS/u\FS$ is the value of $\DDD$ at the
PD-extension $\FS/u\FS=W(k)[[t_1,\ldots,t_d]]\to k[[t_1,\ldots,t_d]]$, 
which is a quotient of $S_r\to R_r$ on which $\Gamma_r$ acts trivially.
\end{proof}

\begin{proposition}
\label{Pr:BT-DF-WinGamma}
The natural functors
\[
\pdiv(R_r) \to
\DF(\Spf R_r) \to
\Win({\u S}\rr)^{\nabla_0,\Gamma_r} \xrightarrow j
\Win({\u S}\rr)^{\nabla_0}
\]
are all equivalences of categories. $($The first functor is contravariant.$)$
\end{proposition}

\begin{proof}
The first arrow is an equivalence by Corollary \ref{Co:BT2DF}.
The second arrow exists by Lemma \ref{Le:DF2WinConnGamma},
and its composition with the forgetful functor $j$ is an equivalence
by Proposition \ref{Pr:DF2WinConn} and Lemma \ref{Le:nabla0}.
It remains to show that $j$ is fully faithful. More precisely we show that for an object 
$(\u M,\nabla_0)$ of $\Win(\u{S}\rr)^{\nabla_0}$ and for $\gamma\in\Gamma_r$ 
there is at most one horizontal automorphism
$\gamma_M:\u M\to\u M$ over the automorphism $\gamma$ of\/ $\u{S}\rr$
such that $\gamma_M$ induces the identity on $M\otimes_{S_r}\FS/u\FS$,
or equivalently on $M\otimes_{S_r}W(k)$; see Remark \ref{Re:strict}.
If the given pair $(\u M,\nabla_0)$ corresponds to $G\in\pdiv(R_r)$ by the composite equivalence,
then $\u M\otimes_{\u S\rr}\u W(k)$ corresponds to the special fibre
$G_k$ by classical Dieudonn\'e theory,
and $\gamma_M$ corresponds to an automorphism of $G$ that induces the identity of $G_k$. 
Thus $\gamma_M$ is unique by the rigidity of $p$-divisible groups;
see \cite[II.3.3.21]{Messing} and its proof.
\end{proof}

\begin{lemma}
\label{Le:ell-ff}
The forgetful functor $\Win(\u S\rr)^{\nabla_0,\Gamma_r}\to\Win(\u S\rr)^{\Gamma_r}$ is fully faithful.
\end{lemma}

\begin{proof}
We have to show that for each $\u M\in\Win(\u S\rr)^{\Gamma_r}$ there is at most one 
$\nabla_0:M\to M\otimes_\FS\Omega_0$ which makes $\u M$ into an object of 
$\Win(\u S\rr)^{\nabla_0,\Gamma_r}$.
The evaluation of $\nabla_0$ at $(u+1)(d/du)$ gives a differential operator
$N:M\to M$ with $N\circ\gamma=\chi(\gamma)\gamma\circ N$ for $\gamma\in\Gamma_r$.
The difference of two choices of $\nabla_0$ gives an $S$-linear map $\Delta:M\to M$
with $\Delta\circ\gamma=\chi(\gamma)\gamma\circ\Delta$.
Let $L=\Frac(W(k)[[t_1,\ldots,t_d]])$ and $S'=L[[u]]$.
Then $\Gamma_r$ acts on $S'$, and we have an injective equivariant homomorphism $S_r\to S'$.
Let $\Delta':M'\to M'$ be the scalar extension of $\Delta$ to $S'$.
If $\Delta\ne 0$ we choose $n\ge 0$ maximal such that $\Delta'(M')\subseteq u^nM'$.
Then $\Delta'$ induces a non-zero $L$-linear map $\bar\Delta':M'/uM'\to u^nM'/u^{n+1}M'$.
Since the group $\Gamma_r$ acts strictly on $M$, it acts trivially on $M'/uM'$
and acts via $\chi^n$ on $u^nM'/u^{n+1}M'$. 
Since for $\gamma\ne 1$ the equation $\chi^{n+1}(\gamma)a=a$ 
has no non-zero solution $a\in L$, it follows that $\bar\Delta'=0$, a contradiction.
Thus $\Delta=0$, and $\nabla_0$ is unique.
\end{proof}

\subsection{Construction of the connection}\label{ConnConst}

We want to show that the forgetful functor of Lemma \ref{Le:ell-ff} is essentially
surjective, i.e.\ that every $\u S\rr$-window with a strict $\Gamma_r$-action carries
a natural connection. To this end, we need some preparations. 
As earlier let $u_0=\varphi^r(u_r)$, which is independent of $r$.
We fix $r\ge 1$ and write $u:=u_r$ and $\FS:=\FS_r$ for simplicity.

\begin{lemma}
\label{Le:FSWinCont}
For $\u M\in\Win(\u\FS\rr)^{\Gamma_r}$ the induced action of\/ $\Gamma_r$ on $M/u_0 M$ is trivial.
\end{lemma}

\begin{proof}
Let $\gamma\in\Gamma_r$.
We show by induction that for $0\le s\le r$ we have $(\gamma-1)M\subseteq\varphi^s(u)M$. 
The case $s=r$ is the assertion of the lemma, while the case $s=0$ is the definition of a strict action.
Let $s\ge 1$ and assume that $(\gamma-1)M\subseteq\varphi^{s-1}(u)M$.
Let
$$
N=\{x\in M\mid \gamma(x)-x\in\varphi^s(u)M\}=
\{x\in M\mid \gamma(x)-x\in E_sM\};
$$
here the second equality holds because $\varphi^s(u)=\varphi^{s-1}(u)\cdot E_s$ 
is a product of coprime factors in the factorial ring $\FS$. For $x\in M$ and $a\in\FS$ we have
\begin{equation}
\label{Eq:gamma-a-x}
(\gamma-1)(ax)=(\gamma-1)(a)\cdot x+\gamma(a)\cdot(\gamma-1)(x).
\end{equation}
Using Lemma \ref{Le:GammaFS} it follows that $N$ is an $\FS$-module.
As $\gamma$ commutes with $\Phi$, the inductive hypothesis implies that $\Phi(M)\subseteq N$. 
Since $\varphi(E_r)M\subseteq \FS\cdot\Phi(M)$ it follows that $\varphi(E_r)M\subseteq N$.
Using \eqref{Eq:gamma-a-x} with $a=\varphi(E_r)$, for each $x\in M$
we get that $\gamma(a)\cdot(\gamma-1)(x)\in E_sM$. 
Since $a\not\in E_s\FS$ and thus $\gamma(a)\not\in E_s\FS$,
it follows that  $(\gamma-1)(x)\in E_sM$ as required.
\end{proof}

We consider the element $t:=\log(1+u_0)\in S_r$,
which exists since $u_0=\varphi^{r-1}(u)E_r$ lies in $\Filone S_r$,
and which is independet of $r$.
A direct calculation using the binomial theorem
to expand $\varphi^n(u_0)=(1+u_0)^{p^n} - 1$
and elementary $p$-adic estimates on binomial coefficients
shows that in $S_r$ we have
$$
t=\lim_{n\to\infty}\frac{\varphi^n(u_0)}{p^n}=u_0\prod_{n\ge 1}\frac{\varphi^n(E_r)}p=u_0\cdot\text{unit}.
$$
We have $\varphi(t)=pt$ and $\gamma(t)=\chi(\gamma)t$ for $\gamma\in\Gamma_0$.
It is easy to see that $u_0^{p-1}\in pS_r$ and thus $t^{p-1}\in pS_r$.

\begin{lemma}
\label{Le:SCont}
For $\gamma\in\Gamma_r$ we have $(\gamma-1)(S_r)\subseteq tS_r$.
\end{lemma}

\begin{proof}
We have $(\gamma-1)(\FS)\subseteq u_0\FS\subseteq tS_r$ by Lemma \ref{Le:GammaFS}.
Assume that for some $a\in\Filone S_r$ we have $\gamma(a)-a=bt$ with $b\in S_r$.
Then $\gamma(a^p)=(a+bt)^p=a^p+pxt+b^pt^p$ for some $x\in S_r$. 
Since $t^{p-1}\in pS_r$ it follows that $(\gamma-1)(a^p/p)\in tS_r$.
Since the divided power envelope of the ideal $E_r\FS\subseteq\FS$ is generated as an $\FS$-algebra
by the successive iterates of $E_r$ under $a\mapsto a^p/p$, the lemma follows.
\end{proof}

\begin{lemma}
\label{Le:StrictM}
Let $\u M\in\Win(\u S\rr)^{\Gamma_r}$.
\begin{enumerate}
\item\label{Le:StrictM1}
For $\gamma\in\Gamma_r$ and $n\ge 0$ we have $(\gamma-1)^{n+1}(M)\in(t,p^r)^{n}tM$.
\item\label{Le:StrictM2}
For $\gamma\in\Gamma_{r+n}$ with $n\ge 0$ we have $(\gamma-1)(M)\subseteq(t,p)^ntM$. 
\item\label{Le:StrictM3}
For each $n\ge 0$ the action of\/ $\Gamma_r$ on $M/p^nM$ has an open kernel.
\newline\noindent
In particular, the action of $\Gamma_r$ on $M$ is continuous.
\end{enumerate}
\end{lemma}

\begin{proof}
By Proposition \ref{Pr:WinGammaFS-S}, $\u M$ comes from an object $\u M'\in\Win(\u\FS\rr)^{\Gamma_r}$,
in particular $M=S_r\otimes_{\FS_r} M'$. 
Using \eqref{Eq:gamma-a-x} with $a\in S_r$ and $x\in M'$, 
Lemmas \ref{Le:FSWinCont} and \ref{Le:SCont} 
give the case $n=0$ of both \eqref{Le:StrictM1} and \eqref{Le:StrictM2}. 
A simple induction, using again \eqref{Eq:gamma-a-x} and the relation $\gamma(t)=\chi(\gamma)t$, gives \eqref{Le:StrictM1} in general.
Since the multiplication
$p:\Gamma_{r+n}\to\Gamma_{r+n+1}$ is bijective and
$$
(\gamma^p-1)=(p+\sum_{i=1}^{p-1}(\gamma^i-1))(\gamma-1),
$$
a similar induction gives \eqref{Le:StrictM2} in general. 
Assertion \eqref{Le:StrictM3} follows from \eqref{Le:StrictM2} since $t^{p-1}\in pS$.
\end{proof}

Following \cite[\S 4.1]{Berger:RepDiff} one can differentiate a continuous action of
$\Gamma_r$ on a finite free $S_r$-module $M$ as follows.
For $\gamma\in\Gamma_r$ sufficiently close to $1$ we have $(\gamma-1)(M)\subseteq pM$,
which implies that for such $\gamma$ the series
$$
\log(\gamma)(x)=\sum_{n\ge 1}(-1)^{n-1}\frac{(\gamma-1)^n(x)}n
$$
with $x\in M$ converges. 
For $a\in\ZZ_p$ we have $\log(\gamma^a)=a\cdot\log(\gamma)$, and thus
$$
N_M:=p^r\frac{\log(\gamma)}{\log(\chi(\gamma))}
$$
is a well-defined map $N_M:M\to M\otimes\QQ$ which is independent of $\gamma$
sufficiently close to $1$.
It is easy to see that
$$
N_M=p^r\lim_{\gamma\to 1}\frac{\gamma-1}{\chi(\gamma)-1}.
$$
It follows that $N_M(ax)=N_S(a)x+aN_M(x)$ for $a\in S_r$ and $x\in M$,
in particular $N_{S_r}$ is a derivation, and $N_M$ is a differential operator over $N_{S_r}$.
Since $N_{S_r}(t)=p^rt$ we must have $N_{S_r}=(1+u)t(d/du)$.
For $\gamma\in\Gamma_r$ we have $N_M\gamma=\gamma N_M$.

In our context the main point is that for strict actions on windows, no denominators occur:

\begin{proposition}
\label{Le:WinSNM}
For $\u M\in\Win(\u S\rr)^{\Gamma_r}$ we have $N_M(M)\in tM$ and $N_M\Phi=\Phi N_M$.
\end{proposition}

\begin{proof}
Since $t^{p-1}\in pS_r$, Lemma \ref{Le:StrictM} \eqref{Le:StrictM1} implies that the series 
$\log(\gamma)$ converges for every $\gamma\in\Gamma_r$ and that $\log(\gamma)(M)\subseteq tM$.
Since the factor $p^r/\log(\chi(\gamma))$ is a $p$-adic unit when 
$\gamma\in\Gamma_r\setminus\Gamma_{r+1}$, the first assertion follows.
The second assertion is clear.
\end{proof}

\begin{corollary}\label{Cor:ForgetNabla}
The forgetful functor $\Win(\u S\rr)^{\nabla_0,\Gamma_r}\to\Win(\u S\rr)^{\Gamma_r}$ 
is an equivalence of categories.
\end{corollary}

\begin{proof}
The functor is fully faithful by Lemma \ref{Le:ell-ff}.
Given $\u M\in\Win(\u S\rr)^{\Gamma_r}$, let $\nabla_0:M\to M\otimes_\FS\Omega_0$
be the connection whose evaluation at $(1+u)t(d/du)$ is the differential operator $N_M$ of
Lemma \ref{Le:WinSNM}; this is well-defined since the image of $N_M$ lies in $tM$.
Since the derivation $(1+u)t(d/du):S_r\to S_r$ commutes with $\Phi$ and with all $\gamma\in\Gamma_r$,
the relations $N_M\Phi=\Phi N_M$ and $N_M\gamma=\gamma N_M$ mean that $\Phi$ and the
action of $\Gamma_r$ are horizontal with respect to $\nabla_0$. Thus $(\u M,\nabla_0)$ is an object of 
$\Win(\u S\rr)^{\nabla_0,\Gamma_r}$, which proves that the functor is essentially surjective.
\end{proof}

Together with Proposition \ref{Pr:BT-DF-WinGamma} we obtain:

\begin{corollary}
\label{Cor:pdiv-WinGamma-r}
\sloppy
There is a contravariant and exact equivalence of categories
$
\pdiv(R_r)\cong\Win(\u S\rr)^{\Gamma_r}
$
that preserves duality.
\qed
\end{corollary}

\subsection{\texorpdfstring{$p$-divisible groups with descent}{p-divisible groups with descent}}

The preceding results can be extended quite formally 
to include the following objects. 
Let $F\subseteq K\subseteq K'$ be finite extensions.
We put $R=\OOO_K[[t_1,\ldots,t_d]]$ and $R'=\OOO_{K'}[[t_1,\ldots,t_d]]$.

\begin{definition}
Let $\pdiv(R')^K$ be the category of $p$-divisible groups $G$ over $R'$ 
equipped with a descent of $G\otimes_{R'}R'[p^{-1}]$ to $R[p^{-1}]$.
If $K'$ is contained in $F_\infty$ we also write $\pdiv(R')^K=\pdiv(R')^{\Gamma_K}$.
\end{definition}

\begin{remark}
\label{Re:Tate}
Assume that $K'/K$ is Galois with Galois group $\bar\Gamma$,
for example $K'\subseteq F_\infty$ and $\bar\Gamma=\Gamma_{K}/\Gamma_{K'}$.
Then $\pdiv(R')^K$ is equivalent to the category of $G\in\pdiv(R')$
that are equipped with isomorphisms
$G_{R'[p^{-1}]}\cong\bar\gamma^*(G_{R'[p^{-1}]})$
for each $\bar\gamma\in\bar\Gamma$ satisfying the  obvious cocycle condition.
By Tate's theorem, these isomorphisms extend uniquely to isomorphisms 
$G\cong\bar\gamma^*(G)$ over $R'$.
\end{remark}

\begin{proposition}
\label{Pr:MainProp}
Assume that $F\subseteq K\subseteq F_r$ with $r\ge 1$, and let $\Gamma=\Gamma_K$. 
There are exact equivalences of categories $($the first one contravariant$)$
\[
\pdiv(R_r)^{\Gamma}\cong\Win(\u S\rr)^{\Gamma}\cong
\Win(\u\FS\rr)^{\Gamma}\cong\BT(\u\FS\rr)^{\Gamma}
\]
that preserve duality.
\end{proposition}

\begin{proof}
The second and third equivalence are
Proposition \ref{Pr:WinGammaFS-S} and Lemma \ref{Le:WinKisinGamma}.
To prove the first equivalence, 
let $G\in\pdiv(R_r)$, and let $\u M(G)\in\Win(\u S\rr)^{\Gamma_r}$ 
be the associated $\Gamma_r$-window given by Corollary \ref{Cor:pdiv-WinGamma-r}.
Let us put $\bar\Gamma=\Gamma/\Gamma_r$.
By Remark \ref{Re:Tate}, giving a descent of $G\otimes R_r[p^{-1}]$ to $R[p^{-1}]$
is equivalent to giving isomorphisms
$w_{\bar\gamma}:G\to\bar\gamma^*(G)$ 
for each $\bar\gamma\in\bar\Gamma$ that satisfy the cocycle condition.
Every $\gamma\in\Gamma$ that lifts $\bar\gamma$ acts on the PD-frame $\u S\rr$,
and we have a natural isomorphism $\gamma^*\u M(G)\cong\u M(\bar\gamma^*G)$. 
Thus $w_{\bar\gamma}$ induces an isomorphism $w_\gamma:\gamma^*(\u M(G))\to\u M(G)$.
In this way, we obtain a bijection between descent data for $G\otimes R_r[p^{-1}]$
relative to $\bar\Gamma$ and actions of $\Gamma$ on $\u M(G)$ that
extend the given action of $\Gamma_r$, and the first equivalence of the proposition
follows from that of Corollary \ref{Cor:pdiv-WinGamma-r}.
\end{proof}

To extend this result further, 
we consider a chain 
$F\subseteq K\subseteq K'\subseteq K''\subseteq F_r$ 
of extensions with $r\ge 1$.
Let $\Gamma''\subseteq\Gamma'\subseteq\Gamma$ 
be the corresponding groups, and let $R''=\OOO_{K''}[[t_1,\ldots,t_d]]$.
We note that for $G\in\pdiv(R'')^\Gamma$, the isomorphisms of Remark \ref{Re:Tate}
give an action of $\Gamma/\Gamma''$ on the special fibre $G_k$. 

\begin{lemma}
\label{Le:descent}
The base change functor $\pdiv(R')^\Gamma\to\pdiv(R'')^\Gamma$ is
fully faithful, and its essential image consists of all $G\in\pdiv(R'')^\Gamma$
for which the action of $\Gamma'/\Gamma''$ on $G_k$ is trivial.
\end{lemma}

\begin{proof}
One easily reduces to the case $K=K'$, which leaves us with the functor
$\pdiv(R)\to\pdiv(R'')^\Gamma$. 
Clearly the functor is fully faithful, and for $G$ in its image the action of
$\Gamma/\Gamma''$ on $G_k$ is trivial.
For a given $p$-divisible group $G_k$ over $k$
let $A$ be its universal deformation ring over $W(k)$. 
A deformation of $G_k$ to $G\in\pdiv(R'')$ corresponds to a homomorphism
of $W(k)$-algebras $h:A\to R''$, 
and there are isomorphisms $G\cong\bar\gamma^*G$
which reduce to the identity of $G_k$ for each $\bar\gamma\in\Gamma/\Gamma''$
if and only if for each $\bar\gamma$ we have $\bar\gamma\circ h=h$.
This means that the image of $h$ lies in the ring of invariants $(R'')^{\Gamma/\Gamma''}=R$,
i.e.\ $G$ is defined over $R$.
\end{proof}

Again let $F\subseteq K\subseteq K'\subseteq F_r$ with $r\ge 1$ and
let $\Gamma=\Gamma_K$ and $\Gamma'=\Gamma_{K'}$.

\begin{proposition}
\label{Pr:Final}
There are exact equivalences of categories 
\[
\pdiv(R')^\Gamma\cong\Win(\u S\rr)^\Gamma_{\Gamma'}
\cong\Win(\u\FS\rr)^\Gamma_{\Gamma'}\cong\BT(\u\FS\rr)^\Gamma_{\Gamma'}
\]
$($the first one contravariant$)$
that preserve duality.
\end{proposition}

\begin{proof}
This is a consequence of Proposition \ref{Pr:MainProp} together with Lemma \ref{Le:descent}
applied with $K''=F_r$.
Namely, let $G\in\pdiv(R_r)^\Gamma$ correspond to the $\Gamma$-window $\u M$
over $\u S_r$.  By Remark \ref{Re:strict},
the action of $\Gamma'$ is trivial on $M/u M$ if and only if it is
trivial on $M\otimes_S W(k)$, which means that $\Gamma'/\Gamma_r$ acts
trivially on $G_k$, i.e.\ that $G$ lies in the image of $\pdiv(R')^\Gamma$. 
This gives the first equivalence. The second and third equivalence are straightforward.
\end{proof}

For $K=K'$ we obtain 
Theorem \ref{Th:Main}:

\begin{corollary}
\label{Co:Final}
The category $\pdiv(R)$ is equivalent to the category 
$\BT(\u\FS\rr)_\Gamma$ of BT 
modules $\u\FM$ over $\u\FS_r$ 
with an action of $\Gamma_K$ which is trivial on $\FM/u\FM$.
\end{corollary}

\begin{remark}
Assume that $F\subseteq\tilde K\subseteq\tilde K'\subseteq F_{\tilde r}$ is another instance
of the above data with $K\subseteq\tilde K$ and $K'\subseteq\tilde K'$ and $r\le\tilde r$.
Under the equivalence of Proposition \ref{Pr:Final},
the base change functor $\pdiv(R')^{\Gamma}\to \pdiv(\tilde R')^{\tilde\Gamma}$
corresponds to the base change functor 
$\BT(\u\FS\rr)^{\Gamma}_{\Gamma'}\to \BT(\u\FS{}_{\tilde r})^{\tilde\Gamma}_{\tilde\Gamma'}$
defined by $(\FM,\varphi)\mapsto(\FM\otimes_{\FS_r}\FS_{\tilde r},\varphi\otimes\varphi)$ 
with the obivous action of $\tilde\Gamma$.
Here $\FS_r$ is viewed as a subring of $\FS_{\tilde r}$ as in Remark \ref{Re:System}.
\end{remark}

\begin{example}
\label{Ex:QZ-Gm}
Let us trace the constructions for the $p$-divisible groups $\QQ_p/\ZZ_p$ and $\hat\GG_m$. 
The $\u S\rr$-window associated to $\QQ_p/\ZZ_p$ is 
$\u S\rr=(S_r,\Filone S_r,R_r,\varphi,\varphi_1)$ with the standard
action of $\Gamma_0$,
which corresponds to the $\u\FS\rr$-window $\u\FS\rr=(\FS_r,E_r\FS_r,\varphi,\varphi')$
and the BT 
module $(\FS_r,\varphi)$, both with the standard action of $\Gamma_0$.
The $\u S\rr$-window associated to $\hat\GG_m$ is $\u S\rr^t=(S_r,S_r,p\varphi,\varphi)$,
and we claim that 
this corresponds to the $\u\FS\rr$-window $\u\FS\rr^t=(\FS_r,\FS_r,\varphi(E_r)\varphi,\varphi)$
with $\gamma\in \Gamma_0$ acting as $x\mapsto\varphi(\lambda_\gamma)\gamma(x)$,
where $\lambda_\gamma$ is defined in \eqref{Eq:lambda_gamma}.
Indeed, the base change of $\u\FS\rr^t$ to $\u S\rr$ is equal to 
$(\u S\rr^t)_c=(S_r,S_r,\varphi(E_r)\varphi,c\varphi)$, and multiplication by $u_0/t$ carries
this window isomorphically onto $\u S\rr^t$. The associated BT 
module is $(\FS_r,E_r\varphi)$ with $\gamma\in\Gamma_0$ acting as 
$x\mapsto\lambda_\gamma\gamma(x)$.
\end{example}

\section{Galois representations}
\label{Se:GalRep}

In this section we relate the equivalences of 
Proposition \ref{Pr:Final} and Corollary \ref{Co:Final}
with the theory of Galois representations.
After some preliminaries on relative period rings in \S\ref{RelPeriod},
we treat the general case in \S\ref{GeneralTate}.  We then specialize to the case $d=0$
in \S\ref{dEqNil}
in order to connect with
the theory of Wach modules and of Kisin--Ren,
which we do in \S\ref{FinEHt}--\ref{Subse:Wachmod}.

Recall that $F=W(k)\otimes\QQ$ for a perfect field $k$ of odd characteristic $p$.
Let $R_0=\OOO_F[[t_1,\ldots,t_d]]$
and let $L$ be an algebraic closure of $F((t_1,\ldots,t_d))$.
For a finite extension $K$ of $F$ inside $L$ we consider $R:=\OOO_K[[t_1,\ldots,t_d]]$
as a subring of $L$.
Let $\bar R$ be the union of all finite $R$-algebras $S$ in $L$
such that $S[p^{-1}]$ is etale over $R[p^{-1}]$.
Let $\G_R=\Gal(\bar R[p^{-1}]/R[p^{-1}])$ and let
$\H_R\subset\G_R$ be the kernel of the $p$-adic cyclotomic character,
thus $\G_R/\H_R=\Gamma_K=\Gal(K(\mu_{p^\infty})/K)$. 

Let $R_{(\infty)}=R[t_1^{p^{-\infty}},\ldots,t_d^{p^{-\infty}}]\subseteq\bar R$ 
be the $R$-algebra generated by 
a chosen compatible system of $p^r$-th roots of the $t_i$.
Let $\G_{R_{(\infty)}}=\Gal(\bar R[p^{-1}]/R_{(\infty)}[p^{-1}])$ and let
$\H_{R_{(\infty)}}\subset\G_{R_{(\infty)}}$ be the kernel of the $p$-adic cyclotomic character,
thus again $\G_{R_{(\infty)}}/\H_{R_{(\infty)}}=\Gamma_K$. 
In the case $d=0$ we have $R=R_{(\infty)}$.

We will always assume that $K\subseteq F_r=F(\mu_{p^r})$ for some $r$, 
which means that  $\H_R=\H_{R_0}$.
The main point will be to recover the Tate module of a $p$-divisible group $G$
over $R$ as a representation of $\G_{R_{(\infty)}}$
from the BT module with $\Gamma_K$-action $\u\FM{}_r(G)$ 
associated to $G$ via crystalline Dieudonn\'e theory by Proposition \ref{Pr:Final}.
This can be done by a variant of Faltings' integral comparison isomorphism 
\cite[Th.\ 7]{Faltings}.

\begin{rem}
Since $G$ is determined by $\u\FM{}_r(G)$, this module also determines
the Tate module $T_p(G)$ as a representation of the full Galois group $\G_R$.
It seems to be inherent to the situation that one directly recovers $T_p(G)$ only
as a representation of $\G_{R_{(\infty)}}$. 
Without the action of $\Gamma_K$ one would even get $T_p(G)$ only as a representation
of $\H_{R_{(\infty)}}$; see \cite[Proposition 8.1]{Kim:FormallySmooth}.
\end{rem}

\subsection{Relative period rings}\label{RelPeriod}

To fix the notation we recall some relative period rings following \cite{Brinon:Rep-relatif}.
For $R$ as above let $\RRR:=\varprojlim(\bar R/p,\varphi)$.
Here $\varphi$ is surjective on $\bar R/p$ by \cite[Prop.\ 2.0.1]{Brinon:Rep-relatif}.
Let $A_{\inf}:=W(\RRR)$, and let 
\[
\theta:A_{\inf}\to\wh{\bar R}=\varprojlim\bar R/p^n
\]
be the unique homomorphism that lifts the natural map $A_{\inf}\to\bar R/p$. 
The ideal $\Ker(\theta)$ is principal,
and an element $\xi=(a_0,a_1,\ldots)\in\Ker(\theta)$ with $a_i\in\RRR$
generates this ideal if and only if $a_1$ is a unit;
see \cite[Prop.\ 5.1.2]{Brinon:Rep-relatif} and its proof.
Let $A_{\cris}$ be the $p$-adic completion of the PD-envelope of $\Ker(\theta)$.
The rings $A_{\inf}$ and $A_{\cris}$ carry a Frobenius $\varphi$ and an action of $\G_{R_0}$.

The chosen system of primitive $p^r$-th roots of unity $\varepsilon^{(r)}\in\bar R$
corresponds to an element $\u\varepsilon$ of $\RRR$.
We put $u_0=[\u\varepsilon]-1\in A_{\inf}$ and $u_r=\varphi^{-r}(u_0)$.
Moreover, for $1\le i\le d$ the chosen compatible system 
of $p^r$-th roots of $t_i$ in $\bar R$ gives an element $\u t{}_i\in\RRR$.
By an abuse of notation we write $t_i=[\u t{}_i]\in A_{\inf}$
and thus consider $\FS_r=W(k)[[u_r,t_1,\ldots,t_d]]$ as a subring of $A_{\inf}$.
Here $\varphi$ and an element $g\in\G_{R_0}$ act by 
$\varphi(1+u_r)=(1+u_r)^p$ and $g(1+u_r)=(1+u_r)^{\chi(g)}$.
Moreover $\varphi(t_i)=t_i^p$, and the group $\G_{(R_0)_{(\infty)}}$ acts trivially on $t_i$.

The restriction of $\theta$ to $\FS_r$ is the homomorphism
$\theta_r:\FS_r\to\OOO_{F_r}[[t_1,\ldots,t_d]]$ defined by $1+u_r\mapsto\varepsilon^{(r)}$
and $t_i\mapsto t_i$.
For $r\ge 1$, the element $E=E_r(u_r)=u_0/\varphi^{-1}(u_0)$ is independent of $r$,
and it generates the ideals $\Ker(\theta)$ and $\Ker(\theta_r)$.
For $r\ge 1$ let $S_r$ be the $p$-adic completion of the PD-envelope of $E\FS_r$. 
The construction of \S \ref{Se:descent} applied to $\FS_r$ and $A_{\inf}$
gives a commutative diagram of frames
\begin{equation}
\label{Eq:4frames}
\xymatrix@M+0.2em{
{\u\FS}\rr \ar[r] \ar[d] & 
{\u A}{}_{\inf}  \ar[d] \\
{\u S}\rr \ar[r] &
{\u A}{}_{\cris}
}
\end{equation}
where in the upper row, the ideals $\Filone$ are generated by $E$
with $\varphi_1'(Ex)=\varphi(x)$,
and in the lower row, the ideals $\Filone$ are the natural PD-ideals with 
$\varphi_1=p^{-1}\varphi$.
The horizontal arrows are strict homomorphisms, and the vertical arrows are
$c$-homomorphisms for $c=\varphi(E)/p$ which induce equivalences
on windows by the descent Proposition \ref{Pr:lambda-crys}.
The Galois group $\G_{(R_0)_{(\infty)}}$ 
acts compatibly on all frames of the diagram, more precisely $g\in\G_{(R_0)_{(\infty)}}$ 
induces strict automorphisms of the lower two frames
and $c_g$-automorphisms of the upper two frames with $c_g=g(\varphi(E))/\varphi(E)$.
The action of $\G_{(R_0)_{(\infty)}}$ on $\u\FS\rr$ and on $\u S\rr$ factors over $\Gamma_F$.
Here the nested system $(\u\FS\rr\to\u S\rr)_{r\ge 1}$ coincides with that
considered in \S \ref{Se:DieuCyc} for $d=0$; see Remark \ref{Re:System}.

\begin{lemma}
\label{Le:Ainf-inv}
We have $A_{\inf}^{\varphi=1}=\ZZ_p$ and $A_{\inf}^{\varphi=E}=u_1\ZZ_p$.
\end{lemma}

\begin{proof}
The first assertion is a consequence of $(\bar R/p)^{\varphi=1}=\FF_p$
and thus $\RRR^{\varphi=1}=\FF_p$, which is proved in \cite[Le.\ 6.2.18]{Brinon:Rep-relatif}.
The second assertion is proved by a variant of the argument of loc.cit. 
Note that $\varphi(u_1)=Eu_1$.
Therefore it suffices to show that $\RRR^{\varphi={\bar E}}$ is equal to
$\bar u_1\FF_p$, where $\bar u_1=\varphi^{-1}(\u\varepsilon-1)\in\RRR$ is the image of $u_1$
and $\bar E=\bar u_1^{p-1}$ is the image of $E$. 
Let $a=(a_0,a_1,\ldots)\in\RRR$ with $a_i\in \bar R/p$ such that $a^p=\bar Ea$.
For $r\ge 0$ let $\varpi_r=(\varepsilon^{(r+1)}-1)^{p-1}\in\OOO_{F_{r+1}}$,
 thus $v_p(\varpi_r)=p^{-r}$. 
We have $a_r^p=\varpi_{r}a_r$ in $\bar R/p$.
Hensel's Lemma implies that for $r\ge 1$ the polynomial $X^p-\varpi_{r}X$ 
has a unique root $\tilde a_r\in\bar R$ with $\tilde a_r\equiv a_r$ mod $p\varpi_{r}^{-1}\bar R$.
Since this polynomial has only $p$ distinct roots in the domain $\bar R$, 
it follows that the image of $a_r$ in $\bar R/p\varpi_r^{-1}$ lies in $\FF_p\bar u_{r+1}$.
Therefore $a\in\bar u_1\FF_p$ in $\RRR$.
\end{proof}

\begin{remark}
\label{Rk:u1-reg}
The element $\bar u_1\in\RRR$ and thus also $u_1\in A_{\inf}$ are non-zero divisors. 
Indeed, $\bar R$ is $p$-torsion free and thus flat over $\O_{F_{\infty}}$.
Thus each $a=(a_0,a_1,\ldots)\in\RRR$ with $a_i\in\bar R/p$
and $\bar u_1a=0$ satisfies $a_i\equiv 0$
mod $p\pi_1^{-1}\bar R$, which implies that $a=0$.
\end{remark}

\subsection{Recovering the Tate module}\label{GeneralTate}

Let $F\subseteq K\subseteq K'\subseteq F_r$ with $r\ge 1$ be given,
and let $R=\OOO_K[[t_1,\ldots,t_d]]$ and $R'=\OOO_{K'}[[t_1,\ldots,t_d]]$
and $R_r=\OOO_{F_r}[[t_1,\ldots,t_d]]$.

For $G\in\pdiv(R')^{\Gamma_K}$ the Tate module $T_p(G)$ is
naturally a $\ZZ_p[\G_R]$-module, 
whose restriction to a $\ZZ_p[\G_{R_{(\infty)}}]$-module
we want to reconstruct from
the BT module over $\FS_r$ with $\Gamma_K$-action $\u\FM{}_r(G)$ associated to $G$ 
by Proposition \ref{Pr:Final}. Since the base change functor
$\pdiv(\OOO_{R'})^{\Gamma_K}\to\pdiv(\OOO_{R_r})^{\Gamma_K}$
changes neither $T_p(G)$ nor $\FM_r(G)$,
for simplicity we may and do assume that $K'=F_r$.
In particular, $\pdiv(\OOO_R)$ is viewed
as a full subcategory of $\pdiv(\OOO_{R_r})^{\Gamma_K}$.

In the following we will use the convention that a BT 
module $\u\FM\in\BT(\u\FS\rr)$ corresponds to the windows $\u\MM\in\Win(\u\FS\rr)$ and
$\u M\in\Win(\u S\rr)$, 
moreover $\u\FM{}_{\inf}$,  $\u\MM{}_{\inf}$,  and $\u M{}_{\cris}$
denote the obvious base change to the other two frames of \eqref{Eq:4frames};
see Lemma \ref{Le:WinKisin} and Proposition \ref{Pr:lambda-crys}.

For $\u\FM\in\BT(\u\FS\rr)^{\Gamma_K}$ as in Definition \ref{Def:KisinGamma} 
we consider the $\ZZ_p[\G_{R_{(\infty)}}]$-module
\[
T_{\inf}^*(\u\FM)=\Hom_{\FS_r,\varphi}(\FM,A_{\inf})
\]
where $g\in\G_{R}$ acts by conjugation.

\begin{lemma}
\label{Le:T*inf}
There are natural $\G_{R_{(\infty)}}$-equivariant isomorphisms
\[
T^*_{\inf}(\u\FM)
\cong\Hom(\u\FM_{\inf},\u A{}_{\inf})
\cong\Hom(\u\MM{}_{\inf},\u A{}_{\inf})
\cong\Hom(\u M{}_{\cris},\u A{}_{\cris}).
\]
where the $\Hom$'s are taken in $\BT(\u A{}_{\inf})$, 
$\Win(\u A{}_{\inf})$, and $\Win(\u A{}_{\cris})$ in that order.
\end{lemma}

\begin{proof}
The first isomorphism is clear, the second isomorphism follows from Lemma \ref{Le:WinKisin},
and the third isomorphism exists because $\u A{}_{\inf}\to\u A{}_{\cris}$ is crystalline by
Proposition \ref{Pr:lambda-crys}.
\end{proof}

Continuing our above notational convention, for $G\in\pdiv(R_r)^{\Gamma_K}$
we write $\u{\FM}_r(G)$, $\u{\MM}_r(G)$, {\em etc.}~for the object of $\BT(\u\FS\rr)^{\Gamma_K}$,
$\Win(\u\FS\rr)^{\Gamma_K}$, {\em etc.}  corresponding to $G$
via Proposition \ref{Pr:MainProp}.
The action of $\Gamma_K$ induces an action of $\G_{R_{(\infty)}}$
on the various base change
modules and windows over the frames of \eqref{Eq:4frames}.
We consider the homomorphism of $\ZZ_p[\G_{R_{(\infty)}}]$-modules
\begin{equation}
	\alpha_G:T_p(G)\to T_{\inf}^*(\u{\FM}_r(G))
	\label{Def:AlphaG}
\end{equation}
defined as the composition
\[
\xymatrix{
{T_p(G)=\Hom_{\bar R}(\QQ_p/\ZZ_p,G)} \ar[r]^-{\u M{}_{\cris}} &
{\Hom(\u M{}_{\cris}(G),\u M{}_{\cris}(\QQ_p/\ZZ_p))
\cong T_{\inf}^*(\u{\FM}_r(G))}
}
\]
using Example \ref{Ex:QZ-Gm} and Lemma \ref{Le:T*inf}.  Let
\begin{equation}
\tilde\alpha_G:T_p(G)\otimes_{\ZZ_p}A_{\inf}\to
\FM_r(G)^t\otimes_{\FS_r}A_{\inf}
\label{Def:alphaGtilde}
\end{equation}
be the $A_{\inf}$-linear map induced by $\alpha_G$, 
where $^t$ refers to the dual module as in Remark \ref{Rk:dual-BT}.

\begin{proposition}
\label{Pr:Comp}
Here $\tilde\alpha_G$ is injective with cokernel annihilated by $u_1$,
and $\alpha_G$ is bijective.
\end{proposition}

This is a variant of \cite[Thm.\ 7]{Faltings} with a similar proof.
See also \cite[\S 5.2]{Kim:FormallySmooth}.

\begin{proof}
We start with the case $G=\hat\GG_m$, 
with associated windows as given in Example \ref{Ex:QZ-Gm}. 
The $\ZZ_p$-module $T_p(\hat\GG_m)$ is generated by $\varepsilon:\QQ_p/\ZZ_p\to\hat\GG_m$, 
which induces $t:(\u A{}_{\cris})^t\to\u A{}_{\cris}$ over $\u A{}_{\cris}$.\footnote{
Let us recall the proof. 
The universal vector extension of $\QQ_p/\ZZ_p$ 
is obtained as the pushout of $\ZZ_p\to\QQ_p\to \QQ_p/\ZZ_p$ by $\ZZ_p\to\GG_a$. 
The given map $\varepsilon:\QQ_p/\ZZ_p\to\hat\GG_m$ over $\OOO_{\C_F}$ lifts uniquely to 
a map $\QQ_p\to\hat\GG_m$ over $A_{\cris}$ with $1\mapsto[\varepsilon]$.
Its restriction to $\ZZ_p$ extends to an algebraic homomorphism $\GG_a\to\hat\GG_m$ 
over $\Spf A_{\cris}$, and $\Lie\GG_a\to\Lie\GG_m$ sends $1=\log_{\GG_a}(1)$
to $t=\log_{\GG_m}([\varepsilon])$.
}
The corresponding homomorphism of $\u A{}_{\inf}$-windows $(\u A{}_{\inf})^t\to\u  A{}_{\inf}$ 
is multiplication by $t\cdot u_0/t=u_0$, and the associated map of BT modules
is $\varphi^{-1}(u_0)=u_1:A_{\inf}\to A_{\inf}$.

For general $G$ let $\rho_G$ the ``dual version'' of $\alpha_G$,
defined as the composite
\[
T_p(G^\vee)=\Hom_{\bar R}(G,\hat\GG_m)
\to\Hom(\u M{}_{\cris}(\hat\GG_m),\u M{}_{\cris}(G))
\cong\Hom((\u A{}_{\inf})^t,\u\FM{}_{\inf}(G))
\]
where the last isomorphism follows from
Lemma \ref{Le:T*inf} applied to the dual windows.
Under the identification of $\Z_p$-modules 
$\Hom_{\Z_p}(T_p(G^\vee),\Z_p)\cong T_p(G)$,
the homomorphism $\rho_G$ gives rise to an $A_{\inf}$-linear map
\[
\tilde\rho_G:\FM_r(G)^t\otimes_{\FS_r}A_{\inf}
\to T_p(G)\otimes_{\ZZ_p}A_{\inf}.
\]
We have a commutative diagram with perfect bilinear vertical maps:
\[
\xymatrix@M+0.2em{
T_p(G)\times T_p(G^\vee) \ar[r]^-{\alpha\times\rho} \ar[d] &
\Hom_{A_{\inf}}(\FM_{\inf},A_{\inf})\times\Hom_{A_{\inf}}(A_{\inf},\FM_{\inf}) \ar[d] \\
T(\hat\GG_m) \ar[r]_-\alpha &
\Hom_{A_{\inf}}(A_{\inf},A_{\inf})
}
\]
which implies that $\tilde\rho_G\circ\tilde\alpha_G$ is equal to $1\otimes u_1$.
This proves the first assertion of the proposition.
Since $u_1$ not a zero divisor in $A_{\inf}$  (see Remark \ref{Rk:u1-reg}) it follows
that $\tilde\alpha_G$ and $\tilde\rho_G$ are injective.
To prove the second assertion we have to consider the Frobenius maps.
Let $\varphi^t:\FM_r(G)^t\to\FM_r(G)^t$ be defined as in Remark \ref{Rk:dual-BT}.
Then $\tilde\alpha_G$ and $\tilde\rho_G$ are homomorphisms of Frobenius modules
\[
(T_p(G)\otimes A_{\inf}, 1\otimes E\varphi)\xrightarrow{\tilde\alpha_G}
(\FM_r(G)^t\otimes_{\FS_r}A_{\inf}, \varphi^t\otimes \varphi)\xrightarrow{\tilde\rho_G}
(T_p(G)\otimes A_{\inf}, 1\otimes \varphi).
\]
If we apply the operator $\{x\mid \varphi(x)=Ex\}$ to these maps,
using Lemma \ref{Le:Ainf-inv} we obtain injective maps
\[
T_p(G)\xrightarrow{\alpha_G} T_{\inf}^*(\u{\FM}_r(G))\to T_p(G)\otimes u_1\Z_p
\]
whose composition is bijective. This proves the Proposition.
\end{proof}

\subsection{\texorpdfstring{The case $d=0$}{The case d=0}}\label{dEqNil}

We now specialize the preceding discussion to the case $d=0$.
We keep our running notation,
so we suppose given $F\subseteq K\subseteq K'\subseteq F_r$ with $r\ge 1$,
and note that since $d=0$ we have 
$R=\O_K$, $R'=\O_{K'}$, $R_r=\O_{F_r}$, and $R_{(\infty)}=R$, 
so $\G_R=\G_{R_{(\infty)}}=\G_K$ is the absolute Galois
group of $K$.  Moreover, the rings $\RRR$, $A_{\inf}$
and $A_{\cris}$ are then the ``classical" period rings of Fontaine,
with $\theta: A_{\cris}\to \O_{\C_K}$ the usual map.
Again, for the sake of simplicity, it is harmless to assume that $K'=F_r$
and to regard $\pdiv(\OOO_K)$ 
as a full subcategory of $\pdiv(\OOO_{F_r})^{\Gamma_K}$, which we do henceforth.

Recall that we view $\FS_r:=\O_F[[u_r]]$ as a $\G_F$ and $\varphi$-stable subring
of $A_{\inf}$ by identifying $u_0=[\u{\varepsilon}]-1$ and $u_r=\varphi^{-r}(u_0)$.
Let $\OOO_{\EEE_r}$ be the $p$-adic completion of $\FS_r[1/u_r]$,
viewed as a subring of $W(\Fr(\RRR))$,
let $\OOO_{\hat\EEE_r^{\nr}}\subseteq W(\Fr(\RRR))$ be the completion of its
maximal unramified extension, and define
\begin{equation}
	\FS_r^{\nr}=\OOO_{\hat\EEE_r^{\nr}}\cap A_{\inf}.
	\label{FSdef}
\end{equation}
We have $\H_F=\Gal(\hat\EEE_r^{\nr}/\EEE_r)$.
For $\u\FM\in\BT(\u\FS\rr)^{\Gamma_K}$, consider the $\ZZ_p[\G_{K}]$-module
\[
T_{\nr}^*(\u\FM)=\Hom_{\FS_r,\varphi}(\FM,\FS_r^{\nr})
\]
where $g\in\G_{K}$ acts by conjugation.

\begin{lemma}
\label{Le:T*nr-inf}
There is a natural $\G_{K}$-equivariant isomorphism
\[
T^*_{\nr}(\u\FM)\cong 
T_{\inf}^*(\u\FM)
\]
and $T^*_{\nr}(\u\FM)$ is a free $\ZZ_p$-module of rank equal to the rank of\/ $\FM$ over $\FS_r$.
\end{lemma}

\begin{proof}
We have 
$\Hom_{\FS_r,\varphi}(\FM,\OOO_{\hat\EEE_r^{\nr}})=\Hom_{\FS_r,\varphi}(\FM,W(\Fr(\RRR)))$,
and this is a free $\ZZ_p$-module of rank equal to the rank of $\FM$;
see \cite[A.1.2.7 and A.2.1.3]{Fontaine}.
Since $W(\Fr(\RRR))/A_{\inf}$ has no non-zero $\varphi$-stable and $\FS_r$-finite submodules
(cf.\ \cite[B.1.8.4]{Fontaine}),
this $\Hom$ module coincides with $T_{\inf}^*(\u\FM)=\Hom_{\FS_r,\varphi}(\FM,A_{\inf})$
and therefore also with $T^*_{\nr}(\u\FM)$ by the very definition $(\ref{FSdef})$ of $\FS^{\nr}_r$.
\end{proof}
  
For  $G\in \pdiv(R_r)^{\Gamma_K}$, the homomorphism $\alpha_G$ of (\ref{Def:AlphaG})
can be viewed as a homomorphism
\[
\tilde\alpha_G:T_p(G)\to T^*_{\nr}(\u\FM{}_r(G))
\]
by Lemma \ref{Le:T*nr-inf}, and this homomorphism is bijective by Proposition \ref{Pr:Comp}.
Moreover, the $\FS$-linear map 
\[
\tilde\alpha_G^{\nr}:T_p(G)\otimes_{\ZZ_p}\FS_r^{\nr}\to \FM_r(G)^t \otimes_{\FS_r} \FS^{\nr}_r
\]
induced by $\tilde\alpha_G^{\nr}$ is injective with cokernel annihilated by $u_1$.
Indeed, this follows from the proof of Proposition \ref{Pr:Comp} with $\FS_r^{\nr}$
in place of $A_{\inf}$, using in addition Lemma \ref{Le:T*nr-inf}.

\subsection{\texorpdfstring{Modules of finite $E$-height}{Modules of finite E-height}}\label{FinEHt}

Let $K$ be a finite extension of $F$ contained in $F_\infty$ and let $r\ge 0$.
We denote by $\Rep_{\ZZ_p}(\G_K)$ the category of free $\ZZ_p$-representations of $\G_K$
and by $\Mod_{\OOO_{\EEE_r}}(\varphi,\Gamma_K)$ the category 
of free etale $(\varphi,\Gamma_K)$-modules over $\OOO_{\EEE_r}$.
By Fontaine \cite{Fontaine}, we have mutually inverse equivalences of categories
\[
\Rep_{\ZZ_p}(\G_K)\underset{T_r}{\overset{D_r}{\rightleftarrows}}
\Mod_{\OOO_{\EEE_r}}(\varphi,\Gamma_K)
\]
defined by $D_r(T)=(T\otimes_{\ZZ_p}\OOO_{\hat\EEE_r^{\nr}})^{\H_F}$
and $T_r(M)=(M\otimes_{\OOO_{\EEE_r}}\OOO_{\hat\EEE_r^{\nr}})^{\varphi=1}$.

\begin{remark}
For $0\le s\le r$ these functors are related as follows. 
Let $i:\OOO_{\EEE_s}\to\OOO_{\EEE_r}$ be the inclusion.
The scalar extension functor
\begin{equation}
\label{Eq:scalarExt}
i^*:\Mod_{\OOO_{\EEE_s}}(\varphi,\Gamma_K)\to
\Mod_{\OOO_{\EEE_r}}(\varphi,\Gamma_K)
\end{equation}
satisfies $T_s\cong T_r\circ i^*$ and $D_r\cong i^*\circ D_s$,
in particular \eqref{Eq:scalarExt} is an equivalence.
Moreover, the Frobenius iterate $\varphi^{r-s}$ of $\OOO_{\EEE_r}$
induces a bijecive homomorphism
\[
\lambda_{r,s}:\OOO_{\EEE_r}\to\OOO_{\EEE_s}
\]
with $i\circ\lambda_{r,s}=\varphi^{r-s}$, 
which induces an isomorphism $\lambda_{r,s}^*\circ D_r\cong D_s$.
\end{remark}

\begin{definition}\label{KRmodDef}
For $r\ge 1$ (so $E\in\FS_r$) we write $\Mod_{\FS_r}(\varphi,\Gamma_K)$ for
the category of finite free 
$\FS_r$-modules $\FM$ equipped with a $\varphi$-linear map 
$\varphi_\FM:\FM\to\FM[E^{-1}]$ that induces an isomorphism $\varphi^*\FM[E^{-1}]\cong\FM[E^{-1}]$,
and with an action of $\Gamma_K$ which commutes with $\varphi_\FM$ and which is finite on 
$\FM/u_r\FM$, {\em i.e.}~an open subgroup acts trivially on this quotient.
\end{definition}

Assume that $F\subseteq K\subseteq K'\subseteq F_r$ with $r\ge 1$.
The main result of Kisin-Ren \cite{KisinRen}, specialised to the case 
where the Lubin-Tate group is the multiplicative group $\hat\GG_m$, 
provides the following equivalence:

\begin{proposition}[Kisin--Ren]
	There is an equivalence $T\rightsquigarrow\FM_r(T)$ 
	between the category of all 
	$T\in\Rep_{\ZZ_p}(\G_K)$ whose restriction to $\G_{K'}$ is crystalline
	and the category of all $\FM\in\Mod_{\FS_r}(\varphi,\Gamma_K)$
	with the property that $\Gamma_{K'}$ acts trivially on $\FM/u_r\FM$.
	Furthermore, there is a natural isomorphism of $(\varphi,\Gamma_K)$-modules 
	\begin{equation}
		\FM_r(T)\otimes_{\FS_r}\OOO_{\EEE_r}\cong D_r(T).
		\label{FMTembedding}
	\end{equation}	
\end{proposition}

Strictly speaking, this is stated in \cite{KisinRen} only when $K=K'$, but the general case
is a formal consequence. Moreover, \cite{KisinRen} works with $\FS_0$ instead of $\FS_r$ 
and constructs a functor $T\rightsquigarrow\FM_{\KR}(T)$ 
valued in the category 
$\Mod_{\FS_0}(\varphi,\Gamma_K)$
of $(\varphi,\Gamma_K)$-modules over 
$\FS_0$
defined
just as in Definition \ref{KRmodDef}, replacing $\FS_r$ with 
$\FS_0$
 and $E = E_r(u_r)$
with 
$E_r(u_0)$. 
This makes no difference to the theory, as $\varphi^r$ 
induces a bijective homomorphism  
$\lambda_{r,0}:\OOO_{\EEE_r}\to\OOO_{\EEE_0}$
carrying $\FS_r$ onto 
$\FS_0$
and sending $E$ to
$E_r(u_0)$,
so by base change induces an equivalence of categories 
$\Mod_{\FS_r}(\varphi,\Gamma_K) \simeq \Mod_{\FS_0}(\varphi,\Gamma_K)$,
and we define $\FM_r(T)$ as the inverse image of $\FM_{\KR}(T)$ under this equivalence;
in other words
\begin{equation}
	\FM_r(T):=\FS_r \otimes_{\lambda_{r,0}^{-1},\FS_0} \FM_{\KR}(T).
	\label{Def:KRMod}
\end{equation}

Moreover, by \cite[Th.\ 0.3]{KisinFcrystal} the category of such $T$ with Hodge-Tate 
weights $0$ and $1$ is equivalent to the category $\pdiv(\OOO_{K'})^{\Gamma_K}$.
Thus the case $d=0$ of
Theorem \ref{Th:Main} is a special case of \cite{KisinRen}.
Let us verify that in this case the modules
of \cite{KisinRen} and of Theorem \ref{Th:Main} are indeed the same up to a duality.

\begin{lemma}
\label{Le:ff}
The scalar extension functor 
$$
\Mod_{\FS_r}(\varphi,\Gamma_K)\to\Mod_{\OOO_{\EEE_r}}(\varphi,\Gamma_K)
$$
is fully faithful.
\end{lemma}

We note that this does not hold without a finite action of $\Gamma_K$. 
For example, $u_1:(\FS_r,E\varphi)\to(\FS_r,\varphi)$
becomes an isomorphism over $\OOO_{\EEE_r}$ but is not an isomorphism.

\begin{proof}
For given $\FM$ and $\tilde\FM\in\Mod_{\FS_r}(\varphi,\Gamma_K)$ we write
$M:=\FM\otimes\OOO_{\EEE_r}$ and $\tilde M:=\tilde\FM\otimes\OOO_{\EEE_r}$. 
The assertion is that every $(\varphi,\Gamma_K)$-module homomorphism 
$f:\tilde M\to M$ maps $\tilde\FM$ into $\FM$.
Since for $0\le r\le s$ we have $\FS_r=\FS_s\cap\OOO_{\EEE_r}$ and since
\eqref{Eq:scalarExt} is an equivalence, we may increase $r$ and thus assume
that $K\subseteq F_r$ and that $\Gamma_{F_r}$ acts trivially on $\FM/u_r\FM$ 
and on $\FM'/u_r\FM'$. 
In that case the assertion follows from \cite[Cor.\ 3.3.8]{KisinRen}.
\end{proof}

\begin{remark}
Lemma \ref{Le:ff} is analogous to \cite[Prop.\ 2.1.12]{KisinFcrystal} and \cite[Prop.\ 3.1]{Caruso},
and one can easily give a direct proof along similar lines.
Namely, for $f:\tilde M\to M$ as above, using \cite[2.1.10]{KisinFcrystal} it follows that 
$\FM'=f(\tilde\FM)+\FM$ is torsion free with $\FM'\otimes\OOO_{\EEE_r}=M$.
One can replace $\tilde\FM$ by $\FM'[1/p]\cap M$ and thus assume that $\tilde M=M$ 
with $f=\id$ and $\FM\subseteq\tilde\FM$. 
In order to show that $\FM=\tilde\FM$ one can pass to the determinant.
Since $\FM$ and $\tilde\FM$ have finite $E$-height, we find $\FM=u_0^n\tilde\FM$ 
for some $n\ge 0$, using for example \cite[Le.\ 2.1.2]{BB}.
The action of $\Gamma_K$ can be finite on $\FM/u_r\FM$ and on $\tilde\FM/u_r\tilde\FM$
only when $n=0$.
\end{remark}

Let
$G\in\pdiv(\OOO_{K'})^{\Gamma_K}$ be given, and let
$\u\FM{}_r(G)\in\BT(\u\FS\rr)^{\Gamma_K}$ be the module associated to it
by Proposition \ref{Pr:Final}.
The dual of the homomorphism $\tilde\alpha_G$ of Proposition \ref{Pr:Comp}
induces a $\varphi$ and $\G_K$-equivariant isomorphism
\[
\FM_r(G)\otimes_{\FS_r}\OOO_{\EEE_r^{\nr}}\cong T_p(G)^\vee\otimes_{\ZZ_p}\OOO_{\EEE_r^{\nr}}.
\]
The invariants under $\H_K=\H_F$ give an isomorphism of $(\varphi,\Gamma_K)$-modules
\begin{equation}
\FM_r(G)\otimes_{\FS_r}\OOO_{\EEE_r}\cong D_r(T_p(G)^\vee).\label{FMGembedding}
\end{equation}

\begin{proposition}\label{Prop:KRModReln}
As submodules of $D_r(T_p(G)^\vee)$ via $(\ref{FMTembedding})$ and $(\ref{FMGembedding})$,
	we have 
\[
\FM_r(G)=\FM_r(T_p(G)^\vee).
\]
\end{proposition}

\begin{proof}
This follows immediately from Lemma \ref{Le:ff}.
\end{proof}

\subsection{Representations of finite height and Wach modules}
\label{Subse:Wachmod}

For $r\ge 1$ and $F\subseteq K\subseteq K'\subseteq F_r$,
the {\em Wach modules} of Berger \cite{Berger} and Berger--Breuil \cite{BB}
provide a variant of the description of Kisin--Ren \cite{KisinRen}
of stable $\Z_p$-lattices in $\G_K$-representations whose restriction to
$\G_{K'}$ is crystalline.  As this variant figures prominently in applications
({\em e.g.}~\cite{BenoisBerger}, 
\cite{BuzzGee}, \cite{CD},
\cite{Gstruct},  \cite{CrysRed},  \cite{LLZ}, \cite{LLZ2}, \cite{LZ}),
for the sake of completeness we now
recall the relation between \cite{KisinRen} and \cite{Berger,BB}.

We keep our running notation,  
but from now on we 
drop the index $r=0$ and write $u=u_0$, $\FS=\FS_0$, etc.
For $T\in\Rep_{\ZZ_p}(\G_K)$ let 
\[
D^+(T)=(T\otimes_{\ZZ_p}\FS^{\nr})^{\H_F}.
\]
It is well-known that this is a free $(\varphi,\Gamma_K)$-module over $\FS$ of rank 
$\le$ the rank of $T$; see Lemma \ref{Le:intersect} below 
applied to $M=D(T)$ and $\FM^{\nr}=T\otimes_{\ZZ_p}\FS^{\nr}$.
In fact, $D^+(T)$ is the maximal $\varphi$-stable finitely generated $\FS$-submodule
of $D(T)$, which is denoted $j_*(D(T))$ in \cite[B.1.4]{Fontaine};
this holds since $\OOO_{\hat\EEE^{\nr}}/\FS^{\nr}$ has no non-zero $\varphi$-stable
finitely generated $\FS^{\nr}$-submodule. 

The representation $T$ is called \emph{of finite height} if the rank of  $D^+(T)$ is equal to the
rank of $T$, which implies that $D^+(T)\otimes_{\FS}\OOO_{\EEE}\cong D(T)$.
By a slight abuse of terminology, we call $T$ \emph{crystabelline} if its restriction to $\G_{F_r}$ 
is crystalline for some $r$.
If $T$ is of finite height and de Rham, then $T$ is crystabelline by \cite[\S A.5]{Wach96}. 
Conversely, crystabelline implies finite height by \cite[Th.\ 2.5.3]{BB} or by \cite{KisinRen}.

Assume that $T$ is of finite height.
Let $V=T[1/p]$ and $D^+(V)=(V\otimes_{\Z_p}\FS_{\nr})^{\H_F}$, thus $D^+(V)=D^+(T)[1/p]$.
To simplify the terminology, an $\FS[1/p]$-module $N\subseteq D^+(V)$ or an
$\FS$-module $\FN\subseteq D^+(T)$
will be called \emph{distinguished} if the module is free of rank $\dim(V)$
and stable under $\Gamma_K$
and if $\Gamma_K$ acts through a finite quotient on $N/uN$ or $\FN/u\FN$.
We note the following elementary fact, see for example \cite[Le.\ II.1.3]{Berger} and its proof.

\begin{lemma}
\label{Le:FN-N}
The set of distinguished submodules $N\subseteq D^+(V)$
is in bijection to the set of distinguished submodules 
$\FN\subseteq D^+(T)$ with $\FN\otimes_\FS\OOO_{\EEE}=D(T)$,
via $\FN=N\cap D^+(T)$ and $N=\FN[1/p]$.
\qed
\end{lemma}

Assume now that $T$ is of finite height and de Rham with non-positive Hodge-Tate weights.
Then by \cite[\S A.5]{Wach96} there is a distinguished submodule $N\subseteq D^+(V)$,
or equivalently $\FN\subseteq D^+(T)$.
There are two standard normalizations which make $N$ and $\FN$ unique.
First, by \cite[Th.\ 3.1.1]{BB} there is a unique distinguished submodule 
$N(V)\subseteq D^+(V)$ such that $D^+(V)/ N(V)$ is annihilated by a power of $u$.
This is the maximal distinguished submodule, and $N(V)$ is stable under $\varphi$.
Explicitly, if $N$ is any distinguished submodule of $D^+(V)$, then
\begin{equation}
\label{Eq:N-N(V)}
		N(V)=D^+(V)\cap N[E_n(u)^{-1}]_{n\ge 1},
\end{equation}
see the proof of \cite[Th.\ 3.1.1]{BB}.
The corresponding module $\FN(T)=N(V)\cap D^+(T)$ is the unique distinguished
submodue of $D^+(T)$ such that $D^+(T)/\FN(T)$ is annihilated by a power of $u$;
note that this condition implies that $\FN\otimes_\FS\OOO_{\EEE}=D(T)$.
We call $N(V)$ and $\FN(T)$
the \emph{Wach modules} of $V$ and of $T$.

Second, assume that $K\subseteq F_r$ and that $T$ becomes crystalline over $\G_{F_r}$ 
for some $r\ge 1$, again with non-positive Hodge-Tate weights. 
Then \cite[Cor.\ 3.3.8]{KisinRen} implies that there is a unique distinguished submodule
$\FM_{\KR}(T)\subseteq D^+(T)$ which is stable under $\varphi$ such that 
$\FM_{\KR}(T)/\varphi^*\FM_{\KR}(T)$ is annihilated by a power of $E_r(u)$
and $\FM_{\KR}(T)\otimes_\FS\OOO_{\EEE}=D(T)$.
 The corresponding module $M_{\KR}(V)=\FM_{\KR}[1/p]$ is the unique distinguished
submodule of $D^+(V)$ which is stable under $\varphi$ such that
$M_{\KR}(T)/\varphi^*M_{\KR}(T)$ is annihilated by a power of $E_r(u)$.
Naturally, $M_{\KR}(V)$ and $\FM_{\KR}(T)$ will be called the \emph{Kisin-Ren} 
modules of $V$ and of $T$.

The relation between Wach modules and Kisin--Ren modules is as follows:

\begin{proposition}\label{Prop:KRWachTranslation}
	Let $V$ be a $p$-adic representation of $\G_K$ with non-positive Hodge--Tate weights
	and restriction to $\G_{F_r}$ that is crystalline.  Then 
	\begin{equation*}
		N(V) = D^+(V)\cap M_{\KR}(V)[E_n(u)^{-1}]_{1\le n\le r}	
	\end{equation*}
	and $M_{\KR}(V)=N_r(V)$, where $N_i=N_i(V)$ is defined recursively by
	\begin{equation*}
		N_{i+1} = \varphi^*(N_i)[E_r(u)^{-1}]\cap N_i,\quad\text{and}\quad N_1:=N(V).
	\end{equation*}
	In particular, for $r=1$ we have $M_{\KR}(V)=N(V)$.
\end{proposition}

\begin{proof}
	Using \cite[Le.\ 2.1.2]{BB}, one can see that $D^+(V)/\FM_{\KR}(V)$ is annihilated
	by a power of $E_1(u)\cdots E_r(u)$, and the asserted formula for $N(V)$
	then follows from \eqref{Eq:N-N(V)}.  On the other hand, the
	quotient $N(V)/\varphi^*N(V)$ is annihilated by a power of $E_1(u)\cdots E_r(u)$
    due to \cite[Cor.\ 3.2.6]{BB}, and the claimed iterative construction of $\FM_{\KR}(V)$
    from $N(V)$ follows from this.
\end{proof}

\begin{remark}
When $T=T_p(G)^{\vee}$ for an object $G$ of $\pdiv(\OOO_{K'})^{\Gamma_K}$,
we have three associated modules: the Wach module $N(T)$, the Kisin-Ren module
$\FM_{\KR}(T)$, and the BT module
$\u\FM{}_r(G)\in\BT(\u\FS{}\rr)^{\Gamma_K}$ of Proposition \ref{Pr:Final}.
By Proposition \ref{Prop:KRModReln} and (\ref{Def:KRMod}), we have
$\FM_r(G)=\FS_r \otimes_{\lambda_{r,0}^{-1},\FS_0} \FM_{\KR}(T)$.
The relation between $\FM_{\KR}(T)$ and $N(T)$ is given by 
Proposition \ref{Prop:KRWachTranslation} together with Lemma \ref{Le:FN-N}.
\end{remark}

\begin{lemma}
\label{Le:intersect}
Let $M$ be a finite free $\OOO_{\EEE}$-module and $\FM^{\nr}$ 
a finite free $\FS^{\nr}$-module, both of rank $d$, together with an isomorphism 
$M\otimes_{\OOO_{\EEE}}\OOO_{\hat\EEE^{\nr}}\cong
\FM^{\nr}\otimes_{\FS^{\nr}}\OOO_{\hat\EEE^{\nr}}$.
Then the $\FS$-module
 $\FM=M\cap\FM^{\nr}$ is free of rank $\le d$. 
If the rank is equal to $d$ then $\FM\otimes_{\FS}\OOO_{\EEE}=M$.
\end{lemma}

\begin{proof}
{\em Cf.}~\cite[Lemme III.3]{Colmez:Crelle99}.
Let us first consider an analogous question modulo $p$.
Let $\bar M=M/p$, a vector space over $\EE=\OOO_{\EEE}/p$,
and $\bar\FM^{\nr}=\FM^{\nr}/p$, a free $\OOO_{\EE^{\sep}}$-module,
and put $\bar\FM=\bar M\cap\bar\FM^{\nr}$, an $\OOO_{\EE}$-module. 
Then $\bar\FM$ contains a basis of $\bar M$, and we claim that $\bar\FM$
is free of rank $d$. If not, then $\bar\FM$ is not finitely generated, and we
find a strictly ascending sequence $N_1\subset N_2\subset\ldots\subset\bar\FM$
of free submodules $N_i$ of rank $d$. By passing to the determinant we see that
this cannot exist, which proves the claim. Now the image of $\FM\to M/p^n$ is isomorphic to 
$\FM_n=\FM/p^n$, and we have $\FM=\varprojlim_n\FM_n$.
Since $\FM_1\subseteq\bar\FM$, the module $\FM_1$ is free of rank $\le d$.
A basis of $\FM_1$ lifts to a basis of $\FM$. The final assertion is clear.
\end{proof}

\bibliography{mybib}
\end{document}